\documentclass{amsart}    

\usepackage{graphicx}
\usepackage{amssymb,amsmath,amsthm}
\usepackage[hidelinks]{hyperref}
\usepackage{mathtools}
\usepackage{bm}
\usepackage{amscd}
\usepackage{stmaryrd}
\usepackage{mathrsfs}
\usepackage{tikz}

\newtheorem{theorem}{Theorem}[section]

\newtheorem{lemma}[theorem]{Lemma}

\newtheorem{corollary}[theorem]{Corollary}
\theoremstyle{definition}

\newtheorem{remark}[theorem]{Remark}
\newtheorem{problem}[theorem]{Problem}

\numberwithin{equation}{section}
\numberwithin{table}{section}
\numberwithin{figure}{section}

\title{Some p-robust a posteriori error estimates based on auxiliary spaces}

\author{Yuwen Li}
\address{School of Mathematical Sciences, Zhejiang University, 866 Yuhangtang Road, Hangzhou, Zhejiang 310058, People's Republic of China.}
\email{liyuwen@zju.edu.cn}
\thanks{
Y. Li was supported  by the National Key R\&D Program of China under grant
2024YFA1012600 and the National Natural Science Foundation of China under grant 12471346.}

\begin{document}

\begin{abstract}
This work develops polynomial-degree-robust (p-robust) equilibrated a posteriori error estimates for $H(\rm curl)$, $H(\rm div)$ and $H(\rm divdiv)$ problems, based on $H^1$ auxiliary space decomposition. The proposed framework employs auxiliary space preconditioning and regular decompositions to decompose the finite element residual into $H^{-1}$ residuals that are further controlled by classical p-robust equilibrated a posteriori error analysis. As a result, we obtain novel and simple p-robust a posteriori error estimates of $H(\rm curl)$/$H(\rm div)$ conforming methods and mixed methods for the biharmonic equation. In addition, we prove guaranteed a posteriori upper error bounds under convex domains or certain boundary conditions. Numerical experiments demonstrate  the effectiveness and p-robustness of the proposed error estimators for the N\'ed\'elec edge element methods and the Hellan--Herrmann--Johnson methods.

\end{abstract}

\maketitle

\markboth{Y.~Li}{p-Robust A Posteriori Error Estimates by Auxiliary Spaces}

\section{Introduction}
Adaptive mesh refinement is a common routine in well-designed numerical methods for resolving singularities, sharp gradients, complex geometry as well as saving computational cost in large-scale numerical simulations. In adaptive Finite Element Methods (FEMs), a posteriori error estimation is the key tool for equi-distributing errors and constructing nearly optimal locally refined meshes. Available a posteriori error estimates of FEMs include explicit residual, superconvergent recovery, dual weighted residual, equilibrated residual, hierarchical basis, functional and solver-based a posteriori error estimators, see, e.g., \cite{Bank1996,AinsworthOden2000,BeckerRannacher2001,BankXu2003a,CarstensenBartels2002,Ainsworth2007,Repin2008,HannukainenStenbergVohralik2012,Verfurth2013,ErnVohralik2015,ErnVohralik2020,Li2018SINUM,BankLi2019,LiShui2025arxiv} for an incomplete list of references in this research area.

This work is concerned with a posteriori error estimation based on equilibrated fluxes by solving local vertex-oriented problems, see \cite{AinsworthOden2000,BraessSchoberl2008,ErnVohralik2015} and references therein.   It has been shown in \cite{BraessPillweinSchoberl2009,ErnVohralik2015,ErnVohralik2020,Chaumont2023} that the efficiency of equilibrated residual error estimators is robust with respect to the polynomial degree (referred to as $p$-\emph{robust}) for classical model problems posed in $H^1$ and $H(\rm curl)$ spaces. The $p$-robustness is an important property in $hp$-AFEMs and real-world simulations by very high order FEMs (cf.~\cite{Schwab1998,Demkowicz2007}). Moreover, equilibrated  residual error analysis often produces \emph{guaranteed} upper error bounds free of unknown hidden constants, leading to fully reliable FEMs under such a posteriori error control.  

Generally speaking, equilibration for vector-valued Sobolev spaces on the de Rham complex is more complicated than the scalar-valued $H^1$ space, see~\cite{BraessSchoberl2008,ChaumontErnVohralik2021}. A great effort (cf.~\cite{Chaumont2023,Chaumont2025}) has been devoted to simplifying equilibration process for $H(\rm curl)$ problems. Recently, differential complexes involving higher order derivatives such as elasticity and divdiv complexes are under intensive investigation, see, e.g.,  \cite{ArnoldAwanouWinther2008,ArnoldHu2021,HuLiangMa2022,ChenHuang2022SINUM,ChenHuang2022MCOM,HuLinZhang2025}. Equilibrated residual error estimators for FEMs posed on those complexes (cf.~\cite{Comodi1989,PechsteinSchoberl2011,GopalakrishnanLedererSchoberl2020}) are expected to be even more difficult and complicated. The ad hoc construction of such a posteriori error estimates renders the theoretical analysis and numerical implementation cumbersome. 

With the help of regular decomposition, we shall develop equilibrated a posteriori error estimates of conforming FEMs posed on the de Rham complex. 
The proposed framework makes use of the regular decomposition to reduce finite element residuals in $H(\rm curl)^*$, $H(\rm div)^*$ into $H^{-1}$ residuals with stability constants controlled by the fictitious space lemma from preconditioning theory. After the reduction, each $H^{-1}$ residual is readily bounded by a $H^1$-type equilibrated residual error estimator. In what follows, we obtain novel equilibrated a posteriori error estimates for the semi-definite curl-curl equation, the positive-definite $H(\rm curl)$ and $H(\rm div)$ problems as well as the indefinite Hodge--Laplace equation, see Section \ref{sect:natural}. As a by-product, numerical implementation of the aforementioned error estimators for $H(\rm curl)$ and $H(\rm div)$ problems is as simple as constructing equilibrated fluxes for Poisson's equations in $H^1$. When the domain is convex, a posteriori error upper bounds for $H(\rm curl)$ problems are guaranteed due to the explicit stability constant. As far as we know, equilibrated residual error  estimators for the semi-definite equation ${\rm curlcurl}u=f$ was extensively studied while the positive-definite one ${\rm curlcurl}u+u=f$ has not been investigated in the literature.

In addition to the de Rham complex, we shall present applications of the proposed auxiliary space approach to FEMs related to the divdiv complex. In this scenario, we employ Helmholtz-type decomposition from divdiv complexes to obtain equivalence between $L^2$ norm error in the stress variable and several $H^{-1}/H^{-2}$ residuals. The $H^{-1}$ residual is then estimated by equilibration as before while the $H^{-2}$ residual boils down to data oscillations. Consequently, we derive novel $p$-robust equilibrated a posteriori error estimates of mixed methods for Poisson's equation and the Hellan--Herrmann--Johnson (HHJ) method for the fourth order equation, see Section \ref{sect:stress}. Under essential boundary conditions, those a posteriori upper error bounds are guaranteed with explicit multiplicative constant. It is noted that \cite{BraessPechsteinSchoberl2020} derived an equilibrated residual error estimator of the $C^0$ discontinuous Galerkin method for biharmonic equations. The lower bound in \cite{BraessPechsteinSchoberl2020} was based on a comparison with explicit residual-type estimators that fail to be $p$-robust. To the best of our knowledge, Section \ref{sect:stress} provides the first provably $p$-robust a posteriori error estimates for the HHJ method.

The rest of the paper is organized as follows. In
Section~\ref{sect:prelim} we provide preliminaries
and basic tools. In Section~\ref{sect:natural} we derive equilibrated a posteriori error estimates for controlling the $H(\rm curl)$ and $H(\rm div)$ norm errors.  Section~\ref{sect:stress}
is devoted to equilibrated residual error estimators for controlling the $L^2$ norm error of stress variables in several mixed FEMs. Section~\ref{sect:numerical} numerically illustrates the effectiveness and robustness of the proposed error estimators for the N\'ed\'elec and HHJ finite element methods. Section \ref{sect:conclusion} is the concluding remark.

\section{Preliminaries}\label{sect:prelim} Let $\Omega\subset\mathbb{R}^d$ be a Lipschitz polygon or polyhedron with $d\in\{2,3\}$. Let $\mathbb{V}=\mathbb{R}^d$ and $\mathbb{S}$ be the space of $d\times d$ symmetric matrices. By $L^2(\Omega,\mathbb{V})=L^2(\Omega)\otimes\mathbb{V}$ and $L^2(\Omega,\mathbb{S})=L^2(\Omega)\otimes\mathbb{S}$ we denote the space of $\mathbb{V}$- and $\mathbb{S}$-valued $L^2$ functions, respectively. Similar notation applies to function spaces with different regularity as well as discrete spaces. By $v_i$ and $\tau_i$ we denote the $i$-th entry of a vector $v$ and the $i$-th row of a matrix $\tau$, respectively.
Given $w\in C^\infty(\Omega)$, $v\in C^\infty(\Omega,\mathbb{V})$, $\tau\in C^\infty(\Omega,\mathbb{S})$, define the following differential operators
\begin{align*}
&{\rm curl}w=\Big(\frac{\partial w}{\partial x_2},-\frac{\partial w}{\partial x_1}\Big),\quad{\rm rot}v=\frac{\partial v_2}{\partial x_1}-\frac{\partial v_1}{\partial x_2},\quad\text{ if }\Omega\subset\mathbb{R}^2,\\
&{\rm curl}v=\Big(\frac{\partial v_3}{\partial x_2}-\frac{\partial v_2}{\partial x_3},\frac{\partial v_1}{\partial x_3}-\frac{\partial v_3}{\partial x_1},\frac{\partial v_2}{\partial x_1}-\frac{\partial v_1}{\partial x_2}\Big),\quad\text{ if }\Omega\subset\mathbb{R}^3,\\
&{\rm symcurl}v=\frac{1}{2}\left\{\begin{pmatrix}
    {\rm curl}v_1\\
    {\rm curl}v_2
\end{pmatrix}+\begin{pmatrix}
    {\rm curl}v_1\\
    {\rm curl}v_2
\end{pmatrix}^\top\right\}\quad\text{ if }\Omega\subset\mathbb{R}^2,\\
&\nabla v=\begin{pmatrix}
    \nabla v_1\\ \vdots\\ \nabla v_d
\end{pmatrix},\quad{\rm div}\tau=\begin{pmatrix}
    {\rm div}\tau_1\\ \vdots\\{\rm div}\tau_d
\end{pmatrix}\quad\text{ if }\Omega\subset\mathbb{R}^d,\\
&{\rm hess}w=\nabla(\nabla w).
\end{align*}
For the de Rham complex, we shall consider function spaces 
\begin{align*}
H_0({\rm curl},\Omega)&=\{v\in L^2(\Omega,\mathbb{V}): {\rm curl}v\in L^2(\Omega,\mathbb{V}),~v\times n=0\text{ on }\partial\Omega\},\\
H_0({\rm div},\Omega)&=\{v\in L^2(\Omega,\mathbb{V}): {\rm div}v\in L^2(\Omega),~v\cdot n=0\text{ on }\partial\Omega\},
\end{align*}
equipped with norms $\|v\|_{H(\rm curl,\Omega)}=(\|v\|^2+\|{\rm curl}v\|^2)^\frac{1}{2}$, $\|v\|_{H(\rm div,\Omega)}=(\|v\|^2+\|{\rm div}v\|^2)^\frac{1}{2}$, respectively.
The spaces $H({\rm curl},\Omega)$ and $H({\rm div},\Omega)$ are obtained by removing the boundary conditions from $H_0({\rm curl},\Omega)$ and $H_0({\rm div},\Omega)$, respectively.
We also need the following algebraic operations (cf.~\cite{ArnoldHu2021,HuLinZhang2025})
 \begin{align*}
{\rm mskw}(v)&=\begin{pmatrix}
    0&v_3&-v_2\\
    -v_3&0&v_1\\
    v_2&-v_1&0
\end{pmatrix}\quad\text{ if }\Omega\subset\mathbb{R}^3,\\
{\rm mskw}(w)&=\begin{pmatrix}
    0&w\\
    -w&0
\end{pmatrix}\quad\text{ if }\Omega\subset\mathbb{R}^2,\\
\iota(w)&=wI_{d\times d}\quad\text{ for a scalar-valued }w.
\end{align*}
Given a vector field $v$ in $\mathbb{R}^2$, $v^\perp=(v_2,-v_1)$ is the rotation of $v$ by degree $\pi/2$. For a tensor field $\sigma$, let $\sigma^\perp$ denote the row-wise rotation of $\sigma$ by degree $\pi/2$. We note that 
\begin{subequations}
\begin{align}
{\rm div}({\rm mskw}(w))&={\rm curl}w,\label{eq:div_mskw}\\
\nabla w&={\rm div}\iota(w),\\
{\rm div}(v^\perp)&={\rm rot}v.\label{eq:div_perp}   
\end{align}  
\end{subequations} 

By $(\bullet,\bullet)_K$ we denote the $L^2$ inner product on a manifold $K$. Let $(\bullet,\bullet)=(\bullet,\bullet)_\Omega$ and  $\|\bullet\|=\|\bullet\|_{L^2(\Omega)}$ denote the global $L^2$ norm. 
Let $\langle\cdot,\cdot\rangle$ denote the duality pairing between a Hilbert $V$ and its dual space $V^*$. The identity operator, denoted by id, acts on a context-dependent appropriate space. Let $H^{-1}(\Omega)=H_0^1(\Omega)^*$ be the dual space of $H_0^1(\Omega)=H_0({\rm grad},\Omega)$. The auxiliary spaces under consideration include $H_0^1(\Omega)$ and
\begin{align*}
\widetilde{H}^1(\Omega)&=\big\{v\in H^1(\Omega): (v,1)=0\big\},\\ \widetilde{H}^2(\Omega)&=\big\{v\in H^2(\Omega): v=0\text{ on }\partial\Omega\big\}.     
\end{align*}
The spaces $H_0^1(\Omega)$ and $\widetilde{H}^1(\Omega)$ are equipped with the inner product $\langle-\Delta\bullet,\bullet\rangle=(\nabla\bullet,\nabla\bullet)$ while $\widetilde{H}^2(\Omega)$ admits the inner product $({\rm hess}\bullet,{\rm hess}\bullet)$. 
 
Let $\mathcal{P}_k$ denote the space of polynomials of degree $\leq k$ in $d$ variables and $\mathcal{P}_k(M)$ be the space of polynomials of degree $\leq k$ on a flat manifold $M$.
Given a conforming triangulation $\mathcal{T}_h$ of $\Omega$, let 
\[
V_h=\big\{v_h\in H^1_0(\Omega): v_h|_T\in\mathcal{P}_1~\forall T\in\mathcal{T}_h\big\}
\]
and $V_h(\mathbb{V})=V_h\otimes\mathbb{V}$ be the vector-valued nodal element space. Let $\{a_i\}_{i\in\mathcal{I}}$ be the set of grid vertices in $\mathcal{T}_h$, $\phi_i$ be the hat function associated to $a_i$, and $\Omega_i={\rm supp}(\phi_i)$ the vertex patch surrounding $a_i$. Let $\mathcal{E}_h$ be the collection of grid faces not contained in $\partial\Omega$. Let $h_T={\rm diam}(T)$ be the diameter of $T\in\mathcal{T}_h$. Let $n$ be the piecewise unit vector field on $\mathcal{E}_h$ with $n_E$ being the unit normal to $E\in\mathcal{E}_h$. In $\mathbb{R}^2$, $t:=-n^\perp$ is the piecewise unit tangent vector field on $\mathcal{E}_h$. We adopt the notation $C_1\lesssim C_2$ provided $C_1\leq C_3C_2$ with $C_3$ being a generic constant independent of the mesh size and polynomial degree. We say $C_1\eqsim C_2$ provided $C_1\lesssim C_2$ and $C_2\lesssim C_1$. 

Given an interior face $E$ shared by elements $T_+$, $T_-$, 
let $\llbracket w\rrbracket_E=w|_{T_+}-w|_{T_-}$ denote the jump of a piecewise smooth function $w$ across the face $E$, where the unit normal $n$ along $E$ is pointing from $T_+$ to $T_-$.
If necessary, an operator $D$ is understood in the distributional sense, while $D_h$ is the piecewise version of $D$ on $\mathcal{T}_h$. For example, given a piecewise smooth vector field $\sigma$, the distributional divergence ${\rm div}\sigma\in H^{-1}(\Omega)$ of $\sigma$ reads
\begin{equation*}
\langle{\rm div}\sigma,v\rangle=({\rm div}_h\sigma,v)+(\llbracket\sigma\cdot n\rrbracket,v)_{\mathcal{E}_h},\quad v\in H_0^1(\Omega).
\end{equation*}

\subsection{Equilibration in H(grad)} For our purpose, we need to rewrite the classical equilibration results (cf.~\cite{BraessPillweinSchoberl2009,ErnVohralik2015,ErnVohralik2020}) as a posteriori error bounds for $H^{-1}$ residuals. Equilibration in $H^{-1}(\Omega)$ relies on the following space:
\begin{align*}
\mathcal{RT}^p_{\rm dG}(\Omega_i)=&\{\tau_h\in L^2(\Omega_i,\mathbb{V}): \tau_h|_T\in\mathcal{P}_kx+\mathcal{P}_k(\mathbb{V})~\text{for each}\\
&~T\in\mathcal{T}_h\text{ with }T\subset\Omega_i, \tau_h\cdot n=0\text{ on }\partial\Omega_i\backslash\partial\Omega\},    
\end{align*}
the space of broken Raviart--Thomas (RT) finite element space on $\mathcal{T}_h\cap\Omega_i$ with vanishing normal trace on $\partial\Omega_i\backslash\partial\Omega$. 
The local conforming RT  space is
\[
\mathcal{RT}_p(\Omega_i)=H({\rm div},\Omega_i)\cap\mathcal{RT}_{\rm dG}^p(\Omega_i).\]
We also need the local space of broken piecewise polynomials
\begin{align*}
\mathcal{P}^{\rm dG}_p(\Omega_i)=&\{v_h\in L^2(\Omega_i): v_h|_T\in\mathcal{P}_p\text{ for each }T\in\mathcal{T}_h\\
&\text{ with }T\subset\Omega_i,~(v_h,1)_{\Omega_i}=0\text{ if }a_i\in\mathring{\Omega}\}.
\end{align*}
Given a residual functional $R\in H^{-1}(\Omega)$, let $R_i\in H^{-1}(\Omega)$ be the localization of $R$ at the vertex $a_i$, i.e., 
\begin{equation*}
\langle R_i,v\rangle:=\langle R,v\phi_i\rangle\quad\forall v\in H_0^1(\Omega).    
\end{equation*} 
Equilibrated residual estimators for $\|R\|_{H^{-1}(\Omega)}$ require solving a series of the following problems surrounding each grid vertex, see \cite{BraessPillweinSchoberl2009}.
\begin{problem}[local mixed FEM]\label{prob:min_sigmaRT}
\begin{align*}
&\sigma_i^R=\arg\min_{\tau\in\mathcal{RT}_{\rm dG}^p(\Omega_i)}\|\tau\|_{L^2(\Omega_i)},\\
&\text{ subject to }{\rm div}\tau=R_i\quad\text{ in }H^{-1}(\Omega).    
\end{align*}
\end{problem}
Assume the residual $R$ has a piecewise polynomial representation, i.e.,
\[
\langle R,v\rangle=\sum_{T\in\mathcal{T}_h}(R_T,v)_T+\sum_{E\in\mathcal{E}_h}(R_E,v)_E,\quad v\in H_0^1(\Omega),
\]
where $R_T$ and $R_E$ are polynomials defined on $T$ and $E$, respectively.
The constraint ${\rm div}\tau=R_i$ means that 
\begin{align*}
{\rm div}\tau&=R_T\phi_i\quad\text{ in each }T\subset\Omega_i,\\
\llbracket\tau\cdot n\rrbracket&=R_E\phi_i\quad\text{ on each }E\subset\overline{\Omega}_i\backslash\partial\Omega.
\end{align*}
The minimizer $\sigma_i^R$ in Problem \ref{prob:min_sigmaRT} can be obtained by solving a hybridized mixed FEM for the Laplacian on $\Omega_i$ whenever the constraint set is non-empty. In the next lemma, we present an equilibrated a posteriori error estimates in terms of residuals, which is essentially equivalent to the result in \cite{BraessPillweinSchoberl2009}.
\begin{lemma}[Equilibrated  residual in $H^{-1}(\Omega)$]\label{lem:equi_residualH1}
Let $R\in H^{-1}(\Omega)$ satisfy $\langle R,\phi_i\rangle=0$ for each vertex $a_i\not\in\partial\Omega$ and admit a piecewise polynomial representation. Assume that Problem \ref{prob:min_sigmaRT} admits a solution. Then 
\[
\Big\|\sum_{i\in\mathcal{I}}\sigma_i^R\Big\|\lesssim\|R\|_{H^{-1}(\Omega)}\leq\Big\|\sum_{i\in\mathcal{I}}\sigma_i^R\Big\|.
\]
\end{lemma}
\begin{proof}
By the definition of Problem \ref{prob:min_sigmaRT}, $\sigma:=\sum_i\sigma_i^R\in L^2(\Omega)$ satisfies ${\rm div}\sigma=R$ in $H^{-1}(\Omega)$ in the distributional sense. Then
\begin{align*}
\|R\|_{H^{-1}(\Omega)}&=\sup_{v\in H_0^1(\Omega), |v|_{H^1(\Omega)}=1}|\langle{\rm div}\sigma,v\rangle|\\
&=\sup_{v\in H_0^1(\Omega), |v|_{H^1(\Omega)}=1}|(\sigma,\nabla v)|\leq\|\sigma\|,
\end{align*}   
which confirms the upper bound.

Conversely, let $u_R\in H_0^1(\Omega)$ be the Riesz representor of $R\in H^{-1}(\Omega)$, i.e., $$(\nabla u_R,\nabla v)=\langle R,v\rangle\quad\forall v\in H^1_0(\Omega).$$ For each vertex $a_i\in\mathring{\Omega}$  or $a_i\in\partial\Omega$, let $V_i=\widetilde{H}^1(\Omega_i)$ or $V_i=H_0^1(\Omega_i)$, respectively. Then we have 
\begin{equation}\label{eq:riVi}
\begin{aligned}
\|R_i\|_{V_i^*}&=\sup_{v\in V_i, |v|_{H^1(\Omega_i)}=1}|\langle R,v\phi_i\rangle|\\
&=\sup_{v\in V_i, |v|_{H^1(\Omega_i)}=1}|(\nabla u_R,\nabla(v\phi_i))|\\
&\leq|u_R|_{H^1(\Omega_i)}|v\phi_i|_{H^1(\Omega_i)}\lesssim|u_R|_{H^1(\Omega_i)},
\end{aligned}
\end{equation}
where a Poincar\'e inequality $\|v\|\lesssim{\rm diam}(\Omega_i)|v|_{H^1(\Omega_i)}$ is used in the last inequality.
Combining \eqref{eq:riVi} with Theorem 7 of \cite{BraessPillweinSchoberl2009} shows that for each index $i\in\mathcal{I}$,
\[
\|\sigma_i^R\|_{L^2(\Omega_i)}\lesssim\|R_i\|_{V_i^*}\lesssim|u_R|_{H^1(\Omega_i)}. 
\]
In what follows, we have
\begin{align*}
\|\sigma\|&\lesssim\sum_{i\in\mathcal{I}}\|\sigma_i^R\|_{L^2(\Omega_i)}\lesssim\sum_{i\in\mathcal{I}}|u_R|_{H^1(\Omega_i)}\\
&\eqsim|u_R|_{H^1(\Omega)}=\|R\|_{H^{-1}(\Omega)}.
\end{align*}
The lower bound is verified.
\end{proof}
In the classical literature, equilibrated residual error estimators are derived from a Prager--Synge identity, while Lemma \ref{lem:equi_residualH1} analyzes the residual in $H^{-1}(\Omega)$ using a simple $L^2$ lifting based on distributional divergence.

\subsection{Operator Preconditioning} To obtain robust or guaranteed  a posteriori upper error bound, we use elementary tools from preconditioning theory to quantify the effect of auxiliary space decomposition. Prior to this work, connections between preconditioning and a posteriori error estimation have been investigated in \cite{LiZikatanov2021CAMWA,LiZikatanov2025mcom,LiShui2025arxiv}.
We say a linear operator $A: V\rightarrow V^*$ is Symmetric and Positive-definite (SPD) provided
\[
\langle Av,v\rangle>0~~~\forall v\in V\quad\&\quad \langle Av,v\rangle=0\Longrightarrow v=0.
\]
Similarly, we say a linear operator $B: V^*\rightarrow V$ is SPD if
\[
\langle r,Br\rangle>0~~~\forall r\in V^*\quad\&\quad \langle r,Br\rangle=0\Longrightarrow r=0.
\] 
The SPD operators $A$ and $B$ induce norms $\|v\|_A=\langle Av,v\rangle^\frac{1}{2}$ and $\|r\|_B=\langle r,Br\rangle^\frac{1}{2}$, respectively. In the following, we provide three operator induced norms on dual spaces as basic examples.
\begin{align*}
\|r\|_{H^{-1}(\Omega)}&=\langle r,A^{-1}r\rangle^\frac{1}{2}=\|r\|_{(-\Delta)^{-1}},~V=H_0^1(\Omega),~  A=-\Delta: V\rightarrow V^*,\\
\|r\|_{H({\rm curl},\Omega)^*}&=\langle r,A^{-1}r\rangle^\frac{1}{2}=\|r\|_{({\rm curlcurl}+{\rm id})^{-1}},~V=H_0({\rm curl},\Omega),~  A={\rm curlcurl}+{\rm id},\\
\|r\|_{H({\rm div},\Omega)^*}&=\langle r,A^{-1}r\rangle^\frac{1}{2}=\|r\|_{(-\nabla{\rm div}+{\rm id})^{-1}},~V=H_0({\rm div},\Omega),~A=-\nabla{\rm div}+{\rm id}.
\end{align*}

The next lemma is a cornerstone of many successful preconditioners for linear iterative solvers (cf.~\cite{Xu1996,HiptmairXu2007,Li2024FoCM}). 
\begin{lemma}[Fictitious Space Lemma]\label{lem:FSP}
Let $\overline{V}$ be a Hilbert space, $\bar{A}: \overline{V}\rightarrow \overline{V}^*$ be a SPD operator.
Assume $\Pi: \overline{V}\rightarrow V$ is a surjective linear operator, and 
\begin{itemize}
\item[{\rm (a)}] For any $\bar{v}\in\overline{V},$ it holds that $\|\Pi\bar{v}\|_A\leq c_0\|\bar{v}\|_{\bar{A}}$,
    \item[{\rm (b)}]
For each $v\in V$, there exists $\bar{v}\in \overline{V}$ such that 
$$\Pi\bar{v}=v,\quad\|\bar{v}\|_{\bar{A}}\leq c_1\|v\|_A.$$
\end{itemize}
Then for $B:=\Pi\bar{A}^{-1}\Pi^*$ we have
\begin{equation*}
    c_0^{-2}\langle r,Br\rangle\leq \langle r,A^{-1}r\rangle\leq c_1^2\langle r,Br\rangle,\quad\forall r\in V^*.
\end{equation*}
\end{lemma}

In \cite{LiZikatanov2025mcom}, Lemma \ref{lem:FSP} was used to derive parameter-robust a posteriori error estimates in $H(\rm curl)$ and $H(\rm div)$. In Sections \ref{sect:natural} and \ref{sect:stress}, Lemma \ref{lem:FSP} is the starting point of our $p$-robust and guaranteed a posteriori error estimates.

\section{Natural norm a posteriori error estimates}\label{sect:natural}
The proposed error estimators in this section rely on the following regular decomposition (cf.~\cite{Hiptmair2002,HiptmairXu2007}).
\begin{lemma}[Regular Decomposition in H(curl)]\label{thm:regular_Hcurl}
For any $v\in H_0({\rm curl},\Omega)$, there exist $\varphi\in H_0^1(\Omega)$ and $z\in H_0^1(\Omega,\mathbb{V}),$ such that 
\begin{align*}
v&=\nabla\varphi+z,\\
|\varphi|_{H^1(\Omega)}^2+|z|_{H^1(\Omega)}^2&\leq C_{\rm reg,c}\|v\|_{H({\rm curl},\Omega)}^2,
\end{align*}
where the constant $C_{\rm reg,c}>0$ depends on the domain $\Omega$.
\end{lemma}

\subsection{H(curl) Elliptic Equation}\label{subsect:Hcurl} First we consider the following positive-definite H(curl) problem
\begin{equation}\label{eq:Hcurl_elliptic}
\begin{aligned}
    {\rm curlcurl}u+u&=f\quad\text{ in }\Omega\subset\mathbb{R}^3,\\
    u\times n&=0\quad\text{ on }\partial\Omega,
\end{aligned}
\end{equation}
and its finite element discretization:
$u_h\in V_h({\rm curl})$ satisfying
\begin{equation}\label{eq:Hcurl_elliptic_FEM}
    ({\rm curl}u_h,{\rm curl}v_h)+(u_h,v_h)=(f,v_h),\quad v_h\in V_h({\rm curl}). 
\end{equation}
Here  $V_h({\rm curl})\subset H_0({\rm curl},\Omega)$ is a N\'ed\'elec finite element space. With $r:=f-{\rm curlcurl}u_h-u_h\in H_0({\rm curl},\Omega)^*$, the following error residual relation holds:  
\begin{equation}\label{eq:uherror_curl}
    \|u-u_h\|_{H({\rm curl},\Omega)}=\|r\|_{H_0({\rm curl},\Omega)^*}=\|r\|_{{(\rm curlcurl}+\text{id})^{-1}},
\end{equation}
where $\text{id}: H_0({\rm curl},\Omega)\rightarrow H_0({\rm curl},\Omega)$ is the identity operator. To derive a posteriori error bound for $\|r\|_{({\rm curlcurl}+\text{id})^{-1}}$, it suffices to construct a preconditioner $B_{\rm c}: H_0({\rm curl},\Omega)^*\rightarrow H_0({\rm curl},\Omega)$ for ${\rm curlcurl}+\text{id}$ via a suitable fictitious space. Let $\overline{V}=H_0^1(\Omega)\times H_0^1(\Omega,\mathbb{V})$ and define $\Pi: \overline{V}\rightarrow V=H_0({\rm curl},\Omega)$ as
\[
\Pi(\varphi,z)=\nabla\varphi+z,\quad(\varphi,z)\in H_0^1(\Omega)\times H_0^1(\Omega,\mathbb{V}).
\]
The fictitious space $\overline{V}$ is naturally equipped with the vector $H^1$ inner product $\langle\bar{A}\bullet,\bullet\rangle$ where $\bar{A}={\rm diag}(-\Delta,-\bm{\Delta}): \overline{V}\rightarrow \overline{V}^*$.
Clearly, $\Pi=(\nabla,I_{\rm c})$ is a bounded operator, where $I_{\rm c}: H_0^1(\Omega,\mathbb{V})\hookrightarrow H_0({\rm curl},\Omega)$ is the inclusion. Lemma \ref{thm:regular_Hcurl} implies that $\overline{V}$, $\Pi$ fullfils the assumption (b) in Lemma \ref{lem:FSP} with $V=H_0({\rm curl},\Omega)$, $A={\rm curlcurl}+\text{id}$. In what follows, we conclude that
\begin{equation}\label{eq:Bc}
\begin{aligned}
    B_{\rm c}&=(\nabla,I_c)\begin{pmatrix}
        (-\Delta)^{-1}&O\\
        O&(-\bm{\Delta})^{-1}
    \end{pmatrix}\begin{pmatrix}
        \nabla^*\\I_{\rm c}^*
    \end{pmatrix}\\
    &=\nabla(-\Delta)^{-1}\nabla^*+I_{\rm c}(-\bm{\Delta})^{-1}I_{\rm c}^*
\end{aligned}
\end{equation}
is spectrally equivalent to ${\rm curlcurl}+\text{id}$. To be precise, 
\begin{equation}\label{eq:rcurl}
\begin{aligned}
&\langle r,B_{\rm c}r\rangle\lesssim\|u-u_h\|_{H({\rm curl},\Omega)}^2=\langle r,({\rm curlcurl}+\textbf{id})^{-1}r\rangle\\
&\leq C_{\rm reg,c}\langle r,B_{\rm c}r\rangle=C_{\rm reg,c}\big(\|\nabla^*r\|^2_{(-\Delta)^{-1}}+\|I_{\rm c}^*r\|^2_{(-\bm{\Delta})^{-1}}\big).
\end{aligned}
\end{equation}

The next theorem establishes a $p$-robust a posteriori error estimate for \eqref{eq:Hcurl_elliptic_FEM} by applying equilibrated flux construction for Poisson's equation to $\|\nabla^*r\|_{(-\Delta)^{-1}}=\|\nabla^*r\|_{H^{-1}(\Omega)}$ and $\|I_{\rm c}^*r\|_{(-\bm{\Delta})^{-1}}=\|I_{\rm c}^*r\|_{H^{-1}(\Omega,\mathbb{V})}$.
\begin{theorem}\label{thm:Hcurl}
Assume $V_h({\rm curl})$ is not the lowest order N\'ed\'elec edge element of the first kind. Let $u_h\in V_h({\rm curl})$ and $f$ are piecewise polynomials of degree $\leq p$. Let $\sigma_i\in\mathcal{RT}_p(\Omega_i)$ and $u_i\in\mathcal{P}_p^{\rm dG}(\Omega_i)$ solve
\begin{equation}\label{eq:sigmai}
\begin{aligned}
    (\sigma_i,\tau)_{\Omega_i}+({\rm div}\tau,u_i)_{\Omega_i}&=((f-u_h)\phi_i,\tau)_{\Omega_i},\quad\tau\in\mathcal{RT}_p(\Omega_i),\\
    ({\rm div}\sigma_i,v)_{\Omega_i}&=((f-u_h)\cdot\nabla\phi_i,v)_{\Omega_i},\quad v\in\mathcal{P}^{\rm dG}_p(\Omega_i).
\end{aligned}
\end{equation}
Let $\tau_i\in\mathcal{RT}_{p+1}(\Omega_i,\mathbb{V})$ and $\tilde{u}_i\in\mathcal{P}^{\rm dG}_{p+1}(\Omega_i,\mathbb{V})$ solve
\begin{equation}\label{eq:taui}
\begin{aligned}
&(\tau_i,\tau)_{\Omega_i}+({\rm div}\tau,\tilde{u}_i)_{\Omega_i}=({\rm mskw}(\phi_i{\rm curl}u_h),\tau)_{\Omega_i},\quad\tau\in\mathcal{RT}_{p+1}(\Omega_i,\mathbb{V}),\\
&({\rm div}\tau_i,v)_{\Omega_i}=((f-u_h)\phi_i+\nabla\phi_i\times{\rm curl}u_h,v)_{\Omega_i},\quad v\in\mathcal{P}^{\rm dG}_{p+1}(\Omega_i,\mathbb{V}).
\end{aligned}
\end{equation}
Then for the $H({\rm curl})$ conforming FEM \eqref{eq:Hcurl_elliptic_FEM}, it holds that 
\begin{align*}
&\Big\|\sum_{i\in\mathcal{I}}\sigma_i-(f-u_h)\phi_i\Big\|^2+\Big\|\sum_{i\in\mathcal{I}}\tau_i-{\rm mskw}(\phi_i{\rm curl}u_h)\Big\|^2\lesssim\|u-u_h\|^2_{H({\rm curl},\Omega)}\\
&\leq C_{\rm reg,c}\Big(\Big\|\sum_{i\in\mathcal{I}}\sigma_i-(f-u_h)\phi_i\Big\|^2+\Big\|\sum_{i\in\mathcal{I}}\tau_i-{\rm mskw}(\phi_i{\rm curl}u_h)\Big\|^2\Big).
\end{align*}
\end{theorem}
\begin{proof}
For each boundary vertex $a_i\in\partial\Omega$, both \eqref{eq:sigmai} and \eqref{eq:taui} have unique solutions due to the inf-sup stability of $\mathcal{RT}_p(\Omega_i)\times\mathcal{P}_p^{\rm dG}(\Omega_i)$. 
For each interior vertex $a_i\in\mathring{\Omega}$, the definition of \eqref{eq:Hcurl_elliptic_FEM} implies 
\begin{equation}\label{eq:compatibility}
\begin{aligned}
((f-u_h)\cdot\nabla\phi_i,1)_{\Omega_i}&=0,\\   
((f-u_h)\phi_i+\nabla\phi_i\times{\rm curl}u_h,e_j)_{\Omega_i}&=0,\quad j=1, 2, 3,
\end{aligned}
\end{equation}
where $e_j\in\mathbb{R}^3$ is the $j$-th unit vector.
The compatibility conditions \eqref{eq:compatibility}   ensure the well-posedness of the mixed methods \eqref{eq:sigmai} and \eqref{eq:taui} surrounding an interior vertex $a_i$.   

Given $v\in H_0^1(\Omega)$, the localization of $\nabla^*r$ is 
\begin{align*}
\langle (\nabla^*r)_i,v\rangle&=\langle f-{\rm curlcurl}u_h-u_h,\nabla(\phi_iv)\rangle=(f-u_h,\nabla(\phi_iv))\\
&=((f-u_h)\cdot\nabla\phi_i,v)+((f-u_h)\phi_i,\nabla v)\\
&=((f-u_h)\cdot\nabla\phi_i,v)-\langle{\rm div}((f-u_h)\phi_i),v\rangle.
\end{align*}
A change of variables implies that Problem  \ref{prob:min_sigmaRT} with $R=\nabla^*r$ is equivalent to
\begin{equation}\label{eq:Hcurl_sigmai}
\begin{aligned}
&\sigma_i=\arg\min_{\tau\in\mathcal{RT}_p(\Omega^h_i)}\|\tau-(f-u_h)\phi_i\|_{L^2(\Omega_i)},\\
&\text{ subject to }{\rm div}\tau=(f-u_h)\cdot\nabla\phi_i,    
\end{aligned}
\end{equation}
which is further equivalent to \eqref{eq:sigmai} due to the assumption $f-u_h\in\mathcal{P}_p^{\rm dG}(\Omega_i,\mathbb{V})$. We point out that $${\rm div}(\sigma_i-(f-u_h)\phi_i)=(\nabla^*r)_i\quad\text{ in }H^{-1}(\Omega).$$
Similarly, the localization of $I_{\rm c}^*r$ is 
\begin{align*}
\langle (I_{\rm c}^*r)_i,v\rangle&=\langle f-{\rm curlcurl}u_h-u_h,\phi_iv\rangle\\
&=((f-u_h)\phi_i,v)-({\rm curl} u_h,{\rm curl}(\phi_i v))\\
&=((f-u_h)\phi_i,v)-({\rm curl}u_h,\nabla\phi_i\times v)-(\phi_i{\rm curl}u_h,{\rm curl}v)\\
&=((f-u_h)\phi_i+\nabla\phi_i\times{\rm curl}u_h,v)-\langle{\rm div}\{{\rm mskw}(\phi_i{\rm curl}u_h)\},v\rangle,
\end{align*}
where we used the algebraic identity \eqref{eq:div_mskw}.
Then a matrix-valued version of Problem  \ref{prob:min_sigmaRT} with a vector-valued residual $R=I_{\rm c}^*r=f-{\rm curlcurl}u_h-u_h\in H^{-1}(\Omega,\mathbb{V})$ reduces to 
\begin{equation}\label{eq:Hcurl_taui}
\begin{aligned}
&\tau_i=\arg\min_{\tau\in\mathcal{RT}_{p+1}(\Omega_i,\mathbb{V})}\|\tau-{\rm mskw}(\phi_i{\rm curl}u_h)\|_{L^2(\Omega_i)},\\
&\text{ subject to }{\rm div}\tau=(f-u_h)\phi_i+\nabla\phi_i\times{\rm curl}u_h, 
\end{aligned}
\end{equation}
which is further  equivalent to \eqref{eq:taui}. It is also noted that $${\rm div}(\tau_i-{\rm mskw}(\phi_i{\rm curl}u_h))=(I_{\rm c}^*r)_i\quad\text{ in }H^{-1}(\Omega,\mathbb{V}).$$

Due to $\nabla V_h\subset V_h({\rm curl})$ and $V_h(\mathbb{V})\subset V_h({\rm curl})$ (because $V_h({\rm curl})$ is not the lowest order edge element space), the two residuals $\nabla^*r$ and $I_{\rm c}^*r$ satisfy the assumption in Lemma \ref{lem:equi_residualH1}, i.e., $\langle\nabla^*r,\phi_i\rangle=\langle I_{\rm c}^*r,\phi_i\rangle=0$ for each interior grid vertex $a_i$. It then follows from Lemma \ref{lem:equi_residualH1} and \eqref{eq:Hcurl_sigmai}, \eqref{eq:Hcurl_taui} that
\begin{equation}\label{eq:gradIcr}
\begin{aligned}
\Big\|\sum_{i\in\mathcal{I}}\sigma_i-(f-u_h)\phi_i\Big\|&\lesssim\|\nabla^*r\|^2_{(-\Delta)^{-1}}\leq\Big\|\sum_{i\in\mathcal{I}}\sigma_i-(f-u_h)\phi_i\Big\|,\\
\Big\|\sum_{i\in\mathcal{I}}\tau_i-{\rm mskw}(\phi_i{\rm curl}u_h)\Big\|&\lesssim\|I_{\rm c}^*r\|^2_{(-\bm{\Delta})^{-1}}\leq\Big\|\sum_{i\in\mathcal{I}}\tau_i-{\rm mskw}(\phi_i{\rm curl}u_h)\Big\|.
\end{aligned}
\end{equation}
Finally, combining  \eqref{eq:rcurl} with \eqref{eq:gradIcr} completes the proof. 
\end{proof}

\begin{remark}
The equilibration procedure used in Theorem \ref{thm:Hcurl} is user-friendly in the sense that the codes developed for equilibrated residual error estimators for $H(\rm grad)$ problems could be reused here. However, Theorem \ref{thm:Hcurl} does not apply to the lowest order edge element because the nodal auxiliary finite element space $V_h(\mathbb{V})$ is not a subspace of $V_h(\rm curl)$ in the lowest order case.  
\end{remark}

\begin{corollary}\label{cor:Hcurl}
    When $\Omega$ is convex, the constant $C_{\rm reg,c}=1$ in Theorem \ref{thm:Hcurl}. In particular, for the $H({\rm curl})$ conforming FEM \eqref{eq:Hcurl_elliptic_FEM}, it holds that
    \begin{align*}
\|u-u_h\|^2_{H({\rm curl},\Omega)}\leq \Big\|\sum_{i\in\mathcal{I}}\sigma_i-(f-u_h)\phi_i\Big\|^2+\Big\|\sum_{i\in\mathcal{I}}\tau_i-{\rm mskw}(\phi_i{\rm curl}u_h)\Big\|^2.
\end{align*}
\end{corollary}
\begin{proof}
Given $v\in H_0({\rm curl},\Omega)$, let $v=\nabla\varphi+{\rm curl}\psi$ be the Helmholtz decomposition of $v$, where $\varphi\in H_0^1(\Omega)$ and $\psi\in H({\rm curl},\Omega)$. Clearly ${\rm curl}v={\rm curl}{\rm curl}\psi$ and 
\begin{equation}\label{eq:L2orthogonal}
\|v\|^2=\|\nabla\varphi\|^2+\|{\rm curl}\psi\|^2.  
\end{equation}
Due to the convexity of $\Omega$, we have (cf.~\cite{GiraultRaviart1986})
\begin{equation}\label{eq:Grisvard}
|{\rm curl}\psi|_{H^1(\Omega)}^2\leq\|{\rm curl}{\rm curl}\psi\|^2+\|{\rm div}{\rm curl}\psi\|^2=\|{\rm curl}v\|^2.
\end{equation}
Combining \eqref{eq:L2orthogonal} with \eqref{eq:Grisvard} confirms that $C_{\rm reg,c}=1$ in Lemma \ref{thm:regular_Hcurl}.
\end{proof}

As far as we know, equilibrated a posteriori error estimates for the SPD $H(\rm curl)$ problem \eqref{eq:Hcurl_elliptic_FEM} have not been considered prior to this work. Instead, the following  semi-definite curl-curl equation (cf.~\cite{BraessSchoberl2008,ChaumontErnVohralik2021,ChaumontVohralik2023,Chaumont2023,Chaumont2025}) often serves as a model problem for deriving equilibrated residual error estimators in $H(\rm curl)$: Find $u\in H_0(\rm curl,\Omega)$ such that 
\begin{equation}\label{eq:curlcurl}
\begin{aligned}
    {\rm curlcurl}u&=f\quad\text{ in }\Omega,\\
    u\times n&=0\quad\text{ on }\partial\Omega.
\end{aligned}
\end{equation}
Under the compatibility condition ${\rm div}f=0$, the solution of \eqref{eq:curlcurl} exists and ${\rm curl} u$ is unique.
The corresponding FEM seeks $u_h\in V_h({\rm curl})$ such that  
\begin{equation}\label{eq:curlcurl_FEM}
({\rm curl}u_h,{\rm curl}v_h)=(f,v_h)\quad\forall v\in V_h({\rm curl}).
\end{equation}
Let $V_h({\rm grad})\subset H_0^1(\Omega)$ be the Lagrange finite element space of degree $\leq p+1$. In practice, it is equivalent to introduce a Lagrange multiplier $\sigma_h\in V_h(\rm grad)$  and consider the saddle-point problem
\begin{equation}\label{eq:curlcurl_FEM_saddle}
    \begin{aligned}
        ({\rm curl} u_h,{\rm curl}v_h)+(v_h,\nabla\sigma_h)&=(f,v_h),\quad v_h\in V_h(\rm curl),\\
        (v_h,\nabla\tau_h)&=0,\quad \tau_h\in V_h(\rm grad).
    \end{aligned}
\end{equation}
It is easy to see  \eqref{eq:curlcurl_FEM_saddle} yields the same ${\rm curl}u_h$ as  \eqref{eq:curlcurl_FEM}.

We are interested in $p$-robust a posteriori error estimate of  $\|{\rm curl}(u-u_h)\|$. The next corollary from Theorem \ref{thm:regular_Hcurl} is useful.
\begin{lemma}\label{cor:regular_Hcurl}
For any $v\in H_0({\rm curl},\Omega)$, there exists $z\in H_0^1(\Omega,\mathbb{V})$ such that 
\begin{align*}
    {\rm curl}v&={\rm curl}z,\\
    |z|_{H^1(\Omega)}&\leq C_{\rm curl}\|v\|_{H({\rm curl},\Omega)},
\end{align*}
where the constant $C_{\rm curl}>0$ depends on the domain $\Omega$.
\end{lemma}

Consider the quotient space $\widetilde{H}_0({\rm curl},\Omega)=H_0({\rm curl},\Omega)/\nabla H_0^1(\Omega)$. We note that ${\rm curlcurl}: \widetilde{H}_0({\rm curl},\Omega)\rightarrow \widetilde{H}_0({\rm curl},\Omega)^*$ is SPD. 
Theorem \ref{thm:regular_Hcurl} implies that $\overline{V}=H_0^1(\Omega,\mathbb{V})$, $\Pi=I_{\rm c}$ fulfills the assumption (b) in Lemma \ref{lem:FSP} with $V=\widetilde{H}_0({\rm curl},\Omega)$, $A={\rm curlcurl}$. In what follows, we conclude that $\widetilde{B}_{\rm c}=I_{\rm c}(-\bm{\Delta})^{-1}I^*_{\rm c}$ is spectrally equivalent to $({\rm curlcurl})^{-1}$. Let $r=f-{\rm curlcurl}u_h\in H_0({\rm curl},\Omega)^*$. We have 
\begin{align*}
\langle r,\widetilde{B}_{\rm c}r\rangle^\frac{1}{2}&\lesssim\|{\rm curl}(u-u_h)\|=\langle r,({\rm curlcurl})^{-1}r\rangle^\frac{1}{2}\\
&\leq C_{\rm curl}\langle r,\widetilde{B}_{\rm c}r\rangle^\frac{1}{2}=C_{\rm curl}\|I_{\rm c}^*r\|_{(-\bm{\Delta})^{-1}}.
\end{align*}
Then by the same a posteriori error control for $\|I_{\rm c}^*r\|_{(-\bm{\Delta})^{-1}}=\|I_{\rm c}^*r\|_{H^{-1}(\Omega,\mathbb{V})}$ as in the proof of Theorem \ref{thm:Hcurl}, we obtain a new $p$-robust a posteriori error estimate for the curl-curl equation.
\begin{theorem}\label{thm:curlcurl}
Assume $V_h({\rm curl})$ is not the lowest order N\'ed\'elec edge element of the first kind. Let $u_h\in V_h({\rm curl})$ and $f$ be piecewise polynomials of degree $\leq p$. Let $\sigma_i\in\mathcal{RT}_{p+1}(\Omega_i,\mathbb{V})$ and $u_i\in\mathcal{P}^{\rm dG}_{p+1}(\Omega_i,\mathbb{V})$ solve
\begin{equation}\label{eq:local_curlcurl}
\begin{aligned}
    (\sigma_i,\tau)_{\Omega_i}+({\rm div}\tau,u_i)_{\Omega_i}&=({\rm mskw}(\phi_i{\rm curl}u_h),\tau)_{\Omega_i},\quad\tau\in\mathcal{RT}_{p+1}(\Omega_i,\mathbb{V}),\\
    ({\rm div}\sigma_i,v)_{\Omega_i}&=(f\phi_i+\nabla\phi_i\times{\rm curl}u_h,v)_{\Omega_i},\quad v\in\mathcal{P}^{\rm dG}_{p+1}(\Omega_i,\mathbb{V}).
\end{aligned}
\end{equation}
Then for \eqref{eq:curlcurl_FEM} it holds that 
\begin{equation*}
\Big\|\sum_{i\in\mathcal{I}}\sigma_i-{\rm mskw}(\phi_i{\rm curl}u_h)\Big\|\lesssim\|{\rm curl}(u-u_h)\|\leq C_{\rm curl}\Big\|\sum_{i\in\mathcal{I}}\sigma_i-{\rm mskw}(\phi_i{\rm curl}u_h)\Big\|.
\end{equation*}
When $\Omega$ is convex, the constant $C_{\rm curl}=1$ and it holds that
    \begin{equation*}
\|{\rm curl}(u-u_h)\|\leq \Big\|\sum_{i\in\mathcal{I}}\sigma_i-{\rm mskw}(\phi_i{\rm curl}u_h)\Big\|.
\end{equation*}
\end{theorem}
\begin{remark}\label{remark:oscillation}
For simplicity of presentation, Theorems \ref{thm:Hcurl} and \ref{thm:curlcurl} assume that the right hand side function $f$ is a piecewise polynomial. In general, the data oscillation $\|f-P_{h,p}f\|$ should be added to those a posteriori error estimates derived in Theorems \ref{thm:Hcurl} and \ref{thm:curlcurl}. Here $P_{h,p}f$ is the $L^2$ projection of $f$ onto the broken piecewise polynomial space of degree $\leq p$. This modification preserves guaranteed a posteriori error upper bounds.
\end{remark}

\subsection{H(div) Elliptic Equation}\label{subsect:Hdiv} The same analysis in Section \ref{subsect:Hcurl} applies to the following H(div) elliptic problem
\begin{align*}
    -\nabla{\rm div}u+u&=f\quad\text{ in }\Omega\subset\mathbb{R}^3,\\
    u\cdot n&=0\quad\text{ on }\partial\Omega,
\end{align*}
discretized by the FEM: Find
$u_h\in V_h({\rm div})$ satisfying
\begin{equation}\label{eq:Hdiv_elliptic_FEM}
    ({\rm div}u_h,{\rm div}v_h)+(u_h,v_h)=(f,v_h),\quad v_h\in V_h({\rm div}). 
\end{equation}
Here  $V_h({\rm div})\subset H_0({\rm div},\Omega)$ is a RT or Brezzi--Douglas--Marini (BDM) finite element space. The residual of \eqref{eq:Hdiv_elliptic_FEM} is $r=f-u_h+\nabla{\rm div}u_h\in H_0({\rm div},\Omega)^*.$

The $H(\rm div)$ elliptic equation is not a hot topic in finite element analysis. Nevertheless, \cite{CasconNochettoSiebert2007} derived an explicit residual error estimator and established convergence of adaptive FEMs for \eqref{eq:Hdiv_elliptic_FEM}.
A posteriori error analysis for \eqref{eq:Hdiv_elliptic_FEM} relies on the next H(div) version of regular decomposition.
\begin{lemma}[Regular Decomposition in H(div)]\label{thm:regular_Hdiv}
For any $v\in H_0({\rm div},\Omega)$, there exist $\varphi\in H_0^1(\Omega,\mathbb{V})$ and $z\in H_0^1(\Omega,\mathbb{V}),$ such that 
\begin{align*}
v&={\rm curl}\varphi+z,\\
|\varphi|_{H^1(\Omega)}^2+|z|_{H^1(\Omega)}^2&\leq C_{\rm reg,d}\|v\|_{H({\rm div},\Omega)}^2,
\end{align*}
where the constant $C_{\rm reg,d}>0$ depends on the domain $\Omega$.
\end{lemma}
Lemma \ref{thm:regular_Hdiv} suggests the fictitious space $\overline{V}=H_0^1(\Omega,\mathbb{V})\times H_0^1(\Omega,\mathbb{V})$ and the transfer operator $\Pi=({\rm curl},I_{\rm d}): \overline{V}\rightarrow V=H_0({\rm div},\Omega)$ given by
\[
\Pi(\varphi,z)={\rm curl}\varphi+z,\quad(\varphi,z)\in H_0^1(\Omega,\mathbb{V})\times H_0^1(\Omega,\mathbb{V}),
\]
where $I_{\rm d}: H_0^1(\Omega,\mathbb{V})\hookrightarrow H_0({\rm div},\Omega)$ is the inclusion.
The space $\overline{V}$ is  equipped with the vector $H^1$ inner product $\bar{A}={\rm diag}(-\bm{\Delta},-\bm{\Delta})$.
Clearly, $\Pi=({\rm curl},I_{\rm d})$ is a bounded operator. It follows from Lemma \ref{thm:regular_Hdiv} that $\overline{V}$, $\Pi$ satisfy the assumption (b) in Lemma \ref{lem:FSP} with $V=H_0({\rm div},\Omega)$, $A=-\nabla{\rm div}+\text{id}$. As a result,
\begin{align*}
    B_{\rm d}&=({\rm curl},I_{\rm d})\begin{pmatrix}
        (-\bm{\Delta})^{-1}&O\\
        O&(-\bm{\Delta})^{-1}
    \end{pmatrix}\begin{pmatrix}
        {\rm curl}^*\\I_{\rm d}^*
    \end{pmatrix}\\
    &={\rm curl}(-\bm{\Delta})^{-1}{\rm curl}^*+I_{\rm d}(-\bm{\Delta})^{-1}I_{\rm d}^*
\end{align*}
is spectrally equivalent to $-\nabla{\rm div}+\text{id}$, i.e.,
\begin{equation}\label{eq:residual_Hdiv}
\begin{aligned}
\|r\|^2_{(-\nabla{\rm div}+\text{id})^{-1}}&=\langle r,({-\nabla\rm div}+\text{id})^{-1}r\rangle\eqsim\langle r,B_{\rm d}r\rangle\\
&=\|{\rm curl}^*r\|^2_{(-\bm{\Delta})^{-1}}+\|I_{\rm d}^*r\|^2_{(-\bm{\Delta})^{-1}}.
\end{aligned}
\end{equation}
For $v\in H_0^1(\Omega,\mathbb{V})$,
direct calculation shows that 
\begin{align*}
\langle ({\rm curl}^*r)_i,v\rangle&=((f-u_h)\times\nabla\phi_i,v)+\langle{\rm div}\{{\rm mskw}((f-u_h)\phi_i)\},v\rangle,\\
\langle (I^*_{\rm d}r)_i,v\rangle&=((f-u_h)\phi_i-(\nabla\phi_i){\rm div}u_h,v)+\langle{\rm div}\iota(\phi_i{\rm div}u_h),v\rangle.
\end{align*}
Combining the previous localization with \eqref{eq:residual_Hdiv} and applying Lemma \ref{lem:equi_residualH1} to estimate $\|{\rm curl}^*r\|_{(-\bm{\Delta})^{-1}}=\|{\rm curl}^*r\|_{H^{-1}(\Omega,\mathbb{V})}$ and $\|I_{\rm d}^*r\|_{(-\bm{\Delta})^{-1}}=\|I_{\rm d}^*r\|_{H^{-1}(\Omega,\mathbb{V})}$, we obtain a $p$-robust a posteriori error estimate for the H(div) elliptic problem \eqref{eq:Hcurl_elliptic_FEM}. The proof is similar to Theorem \ref{thm:Hcurl} and is skipped.
\begin{theorem}\label{thm:Hdiv}
Assume $u_h\in V_h({\rm div})$ and $f$ are piecewise polynomials of degree $\leq p$ and $V_h(\mathbb{V})\subset V_h({\rm div})$. Let $\sigma_i\in\mathcal{RT}_p(\Omega_i,\mathbb{V})$ and $u_i\in\mathcal{P}^{\rm dG}_p(\Omega_i,\mathbb{V})$ solve
\begin{align*}
    (\sigma_i,\tau)_{\Omega_i}+({\rm div}\tau,u_i)_{\Omega_i}&=-({\rm mskw}((f-u_h)\phi_i),\tau)_{\Omega_i},\quad\tau\in\mathcal{RT}_p(\Omega_i,\mathbb{V}),\\
    ({\rm div}\sigma_i,v)_{\Omega_i}&=((f-u_h)\times\nabla\phi_i,v)_{\Omega_i},\quad v\in\mathcal{P}^{\rm dG}_p(\Omega_i,\mathbb{V}).
\end{align*}
Let $\tau_i\in\mathcal{RT}_{p+1}(\Omega_i,\mathbb{V})$ and $\tilde{u}_i\in\mathcal{P}^{\rm dG}_{p+1}(\Omega_i,\mathbb{V})$ solve
\begin{align*}
    (\tau_i,\tau)_{\Omega_i}+({\rm div}\tau,\tilde{u}_i)_{\Omega_i}&=-(\iota(\phi_i{\rm div}u_h),\tau)_{\Omega_i},\quad\tau\in\mathcal{RT}_{p+1}(\Omega_i,\mathbb{V}),\\
    ({\rm div}\tau_i,v)_{\Omega_i}&=((f-u_h)\phi_i-(\nabla\phi_i){\rm div}u_h,v)_{\Omega_i},\quad v\in\mathcal{P}^{\rm dG}_{p+1}(\Omega_i,\mathbb{V}).
\end{align*}
Then for \eqref{eq:Hdiv_elliptic_FEM} it holds that 
\begin{align*}
&\Big\|\sum_{i\in\mathcal{I}}\sigma_i+{\rm mskw}((f-u_h)\phi_i)\Big\|^2+\Big\|\sum_{i\in\mathcal{I}}\tau_i+\iota(\phi_i{\rm div}u_h)\Big\|^2\lesssim\|u-u_h\|^2_{H({\rm div},\Omega)}\\
&\leq C_{\rm reg,d}\Big(\Big\|\sum_{i\in\mathcal{I}}\sigma_i+{\rm mskw}((f-u_h)\phi_i)\Big\|^2+\Big\|\sum_{i\in\mathcal{I}}\tau_i+\iota(\phi_i{\rm div}u_h)\Big\|^2\Big).
\end{align*}
When $\Omega$ is convex, the constant $C_{\rm reg,d}=1$ and it holds that
    \begin{equation*}
\|u-u_h\|_{H({\rm div},\Omega)}^2\leq \Big\|\sum_{i\in\mathcal{I}}\sigma_i+{\rm mskw}((f-u_h)\phi_i)\Big\|^2+\Big\|\sum_{i\in\mathcal{I}}\tau_i+\iota(\phi_i{\rm div}u_h)\Big\|^2.
\end{equation*}
\end{theorem}

\subsection{Hodge--Laplace equation} The analysis in Sections \ref{subsect:Hcurl} and \ref{subsect:Hdiv} directly applies to indefinite problems. For instance, we consider the Hodge--Laplace equation
\begin{align*}
{\rm curlcurl}u-\nabla{\rm div}u&=f\quad\text{in }\Omega,\\
    u\cdot n=0,\quad n\times{\rm curl}u&=0\quad\text{ on }\partial\Omega,
\end{align*}
and its variational formulation: $(\sigma,u)\in H_0({\rm curl},\Omega)\times H_0({\rm div},\Omega)$ such that
\begin{subequations}\label{eq:var_HL}
    \begin{align}
    (\sigma,\tau)-({\rm curl}\tau,u)&=0,\quad\tau\in H_0({\rm curl},\Omega),\\
    ({\rm curl}\sigma,v)+({\rm div}u,{\rm div}v)&=(f,v),\quad v\in H_0({\rm div},\Omega).
\end{align}
\end{subequations}
The corresponding FEM seeks $(\sigma_h,u_h)\in V_h({\rm curl})\times V_h({\rm div})$ such that 
\begin{subequations}\label{eq:FEM_HL}
    \begin{align}
    (\sigma_h,\tau)-({\rm curl}\tau,u_h)&=0,\quad\tau\in V_h({\rm curl}),\\
    ({\rm curl}\sigma_h,v)+({\rm div}u_h,{\rm div}v)&=(f,v),\quad v\in V_h({\rm div}).
\end{align}
\end{subequations}

The Hodge--Laplacian and its finite element discretizations are the model problem in finite element exterior calculus (cf.~\cite{ArnoldFalkWinther2006,Arnold2018}). The well-posedness of \eqref{eq:FEM_HL} is confirmed in \cite{ArnoldFalkWinther2006} provided $V_h({\rm curl})$ and $V_h({\rm div})$ are of equal polynomial degree or the degree of $V_h({\rm curl})$ is one order higher than $V_h({\rm div})$.
\begin{corollary}\label{cor:Hodge_Laplace}
Assume $\sigma_h\in V_h({\rm curl})$ is a piecewise polynomial of degree $\leq p+1$, $u_h\in V_h({\rm div})$ and $f$ are piecewise polynomials of degree $\leq p$ in \eqref{eq:FEM_HL}. In addition, assume $V_h(\mathbb{V})\subset V_h({\rm curl})$ and $V_h(\mathbb{V})\subset V_h({\rm div})$. Let $\sigma_i^1\in\mathcal{RT}_{p+1}(\Omega_i)$ and $u_i^1\in\mathcal{P}^{\rm dG}_{p+1}(\Omega_i)$ solve
\begin{align*}
    (\sigma_i^1,\tau)_{\Omega_i}+({\rm div}\tau,u^1_i)_{\Omega_i}&=-(\phi_i\sigma_h,\tau)_{\Omega_i},\quad\tau\in\mathcal{RT}_{p+1}(\Omega_i),\\
    ({\rm div}\sigma^1_i,v)_{\Omega_i}&=(\sigma_h\cdot\nabla\phi_i,v)_{\Omega_i},\quad v\in\mathcal{P}^{\rm dG}_{p+1}(\Omega_i).
\end{align*}
Let $\sigma_i^2\in\mathcal{RT}_{p+2}(\Omega_i,\mathbb{V})$ and $u_i^2\in\mathcal{P}^{\rm dG}_{p+2}(\Omega_i,\mathbb{V})$ solve
\begin{align*}
    (\sigma_i^2,\tau)_{\Omega_i}+({\rm div}\tau,u^2_i)_{\Omega_i}&=({\rm mskw}(\phi_iu_h),\tau)_{\Omega_i},\quad\tau\in\mathcal{RT}_{p+2}(\Omega_i,\mathbb{V}),\\
    ({\rm div}\sigma_i^2,v)_{\Omega_i}&=(\phi_i\sigma_h-u_h\times\nabla\phi_i,v)_{\Omega_i},\quad v\in\mathcal{P}^{\rm dG}_{p+2}(\Omega_i,\mathbb{V}).
\end{align*} 
Let $\sigma_i^3\in\mathcal{RT}_p(\Omega^h_i,\mathbb{V})$ and $u_i^3\in\mathcal{P}^{\rm dG}_p(\Omega_i,\mathbb{V})$ solve
\begin{align*}
    (\sigma_i^3,\tau)_{\Omega_i}+({\rm div}\tau,u^3_i)_{\Omega_i}&=-({\rm mskw}\{(f-{\rm curl}\sigma_h)\phi_i\},\tau)_{\Omega_i},\quad\tau\in\mathcal{RT}_p(\Omega_i,\mathbb{V}),\\
    ({\rm div}\sigma^3_i,v)_{\Omega_i}&=((f-{\rm curl}\sigma_h)\times\nabla\phi_i,v)_{\Omega_i},\quad v\in\mathcal{P}^{\rm dG}_p(\Omega_i,\mathbb{V}).
\end{align*}
Let $\sigma^4_i\in\mathcal{RT}_{p+1}(\Omega_i,\mathbb{V})$ and $u^4_i\in\mathcal{P}^{\rm dG}_{p+1}(\Omega_i,\mathbb{V})$ solve
\begin{align*}
    (\sigma^4_i,\tau)_{\Omega_i}+({\rm div}\tau,u^4_i)_{\Omega_i}&=-(\iota(\phi_i{\rm div}u_h),\tau)_{\Omega_i},\quad\tau\in\mathcal{RT}_{p+1}(\Omega_i,\mathbb{V}),\\
    ({\rm div}\sigma^4_i,v)_{\Omega_i}&=((f-{\rm curl}\sigma_h)\phi_i-({\rm div} u_h)\nabla\phi_i,v)_{\Omega_i},\quad v\in\mathcal{P}^{\rm dG}_{p+1}(\Omega_i,\mathbb{V}).
\end{align*}
Then for \eqref{eq:FEM_HL} it holds that 
\begin{align*}
&\|\sigma-\sigma_h\|_{H({\rm curl},\Omega)}+\|u-u_h\|_{H({\rm div},\Omega)}\\
&\eqsim\Big\|\sum_{i\in\mathcal{I}}\sigma_i^1+\phi_i\sigma_h\Big\|+\Big\|\sum_{i\in\mathcal{I}}\sigma_i^2-{\rm mskw}(\phi_iu_h)\Big\|\\
&+\Big\|\sum_{i\in\mathcal{I}}\sigma_i^3+{\rm mskw}\{(f-{\rm curl}\sigma_h)\phi_i\}\Big\|+\Big\|\sum_{i\in\mathcal{I}}\sigma_i^4+\iota(\phi_i{\rm div}u_h)\Big\|.
\end{align*}
\end{corollary}
\begin{proof}
Let $r_\sigma=\sigma_h-{\rm curl}u_h\in H_0({\rm curl},\Omega)^*$, $r_u=f+\nabla{\rm div}u_h-{\rm curl}\sigma_h\in H_0({\rm div},\Omega)^*$.
The continuous inf-sup stability of \eqref{eq:var_HL} implies 
\begin{equation}
\begin{aligned}
&\|\sigma-\sigma_h\|_{H({\rm curl},\Omega)}+\|u-u_h\|_{H({\rm div},\Omega)}\\
&\eqsim\|r_\sigma\|^2_{({\rm curlcurl}+{\rm id})^{-1}}+\|r_u\|^2_{(-\nabla{\rm div}+{\rm id})^{-1}}.
\end{aligned}
\end{equation}
It then follows from \eqref{eq:residual_Hdiv} and \eqref{eq:rcurl} that
\begin{align*}
&\|\sigma-\sigma_h\|_{H({\rm curl},\Omega)}+\|u-u_h\|_{H({\rm div},\Omega)}\\
&\eqsim\|\nabla^*r_\sigma\|_{(-\Delta)^{-1}}+\|I_{\rm c}^*r_\sigma\|_{(-\bm{\Delta})^{-1}}+\|{\rm curl}^*r_u\|_{(-\bm{\Delta})^{-1}}+\|I_{\rm d}^*r_u\|_{(-\bm{\Delta})^{-1}}.
\end{align*}
The rest of the proof is a combination of the proofs of Theorems \ref{thm:Hcurl} and \ref{thm:Hdiv}.
\end{proof}
For the Hodge--Laplace equation,  explicit residual error estimators were derived in \cite{DemlowHirani2014,Li2019SINUM} and implicit sub-domain residual error estimators were derived in \cite{LiZikatanov2025mcom}. Corollary \ref{cor:Hodge_Laplace} provides the first $p$-robust a posteriori error estimate of  mixed FEMs for the Hodge--Laplacian.

\section{A Posteriori Error Estimates of Stress Variables}\label{sect:stress}
The previous section was concerned with a posteriori error control in terms of natural norms under which FEMs are inf-sup stable. In this section, we employ the auxiliary space approach to derive $p$-robust and guaranteed a posteriori error estimates for controlling the $L^2$ norm errors in the stress variable of several mixed FEMs. For simplicity, we assume $\Omega\subset\mathbb{R}^2$ is a two-dimensional polygon. 

\subsection{Mixed FEM for the Laplacian}\label{subsect:mixPoisson} To illustrate the underlying idea, first we consider the Poisson equation in mixed form:
\begin{align*}
\sigma&=\nabla u\quad\text{ in }\Omega,\\
{\rm div}\sigma&=-f\quad\text{ in }\Omega,\\
    \sigma\cdot n&=0\quad\text{ on }\partial\Omega.
\end{align*}
Let $V_h({\rm div})\subset H_0(\rm div,\Omega)$ be the space of RT or BDM finite element space. Let $U_h$ is the space of broken piecewise polynomials subject to zero mean value constraint.
The corresponding FEM seeks $\sigma_h\in V_h({\rm div})$ and $u_h\in U_h$ such that 
\begin{subequations}\label{def:MFEM_Poisson}
    \begin{align}
    (\sigma_h,\tau_h)+({\rm div}\tau_h,u_h)&=0,\quad\tau\in V_h({\rm div}),\label{def:MFEM_Poisson1}\\
    ({\rm div}\sigma_h,v_h)&=-(f,v_h),\quad v_h\in U_h.
\end{align}
\end{subequations}
Here $V_h(\rm div)$ and $U_h$ are paired such that \eqref{def:MFEM_Poisson} fulfills a discrete inf-sup condition. Analysis in this section relies on the following well-known Helmholtz decomposition as well as its tensor-valued version.
\begin{lemma}[Helmholtz Decomposition of Vector Fields]\label{lem:Helmholtz_vector}
Let $\Omega\subset\mathbb{R}^2$  be a simply-connected Lipschitz domain. For each $\tau\in L^2(\Omega,\mathbb{V})$, there exist $\varphi\in H_0^1(\Omega)$ and $\psi\in\widetilde{H}^1(\Omega)$ such that 
\begin{align*}
\tau={\rm curl}\varphi\oplus\nabla\psi,
\end{align*}
where ${\rm curl}\varphi$ and $\nabla \psi$ are orthogonal with respect to $(\bullet,\bullet)$.
\end{lemma}

Now we are in a position to present a new equilibrated residual error estimator for \eqref{def:MFEM_Poisson} based on the auxiliary space approach. 
\begin{theorem}\label{thm:mixPoisson}
Assume $\Omega$ is a simply-connected polygonal domain. Let $\sigma_h\in V_h({\rm div})$ be a piecewise polynomial of degree $\leq p$. For each $i\in\mathcal{I}$, let $\sigma_i\in\mathcal{RT}_p(\Omega^h_i)$ and $u_i\in\mathcal{P}^{\rm dG}_p(\Omega_i)$ solve
\begin{equation}\label{eq:local_MFEM}
\begin{aligned}
    (\sigma_i,\tau)_{\Omega_i}+({\rm div}\tau,u_i)_{\Omega_i}&=-(\sigma_h^\perp\phi_i,\tau)_{\Omega_i},\quad\tau\in\mathcal{RT}_p(\Omega_i),\\
    ({\rm div}\sigma_i,v)_{\Omega_i}&=(\sigma_h\cdot{\rm curl}\phi_i,v)_{\Omega_i},\quad v\in\mathcal{P}^{\rm dG}_p(\Omega_i).
\end{aligned}
\end{equation}
Then for the mixed method \eqref{def:MFEM_Poisson} it holds that 
\begin{equation*}
\|\sigma-\sigma_h\|^2\leq \Big\|\sum_{i\in\mathcal{I}}\sigma_i+\sigma_h^\perp\phi_i\Big\|^2+\|\gamma h(f-f_h)\|^2.
\end{equation*}
where  $f_h$ is the $L^2$ projection of $f$ onto $U_h$, and  $\gamma$ is a piecewise constant with $\gamma_T=h_T^{-1}\max_{v\neq0}\|v-v_h\|_{L^2(T)}/|v|_{H^1(T)}$ for $T\in\mathcal{T}_h$. In addition,
\begin{equation*}
\Big\|\sum_{i\in\mathcal{I}}\sigma_i+\sigma_h^\perp\phi_i\Big\|\lesssim\|\sigma-\sigma_h\|.
\end{equation*}
\end{theorem}
\begin{proof}
For each interior vertex $a_i\in\mathring{\Omega}$, the definition of \eqref{def:MFEM_Poisson1} with $\tau_h={\rm curl}\phi_i$ implies the compatibility condition $(\sigma_h\cdot{\rm curl}\phi_i,1)_{\Omega_i}=0$,
ensuring the well-posedness of the local problem \eqref{eq:local_MFEM} surrounding $a_i$.

Let $\overline{V}=H_0^1(\Omega)\times \widetilde{H}^1(\Omega)$ and $\Pi=({\rm curl},\nabla): \overline{V}\rightarrow L^2(\Omega,\mathbb{V})$ be given by
\[
\Pi(\varphi,\psi)={\rm curl}\varphi+\nabla \psi,\quad(\varphi,\psi)\in H_0^1(\Omega)\times\widetilde{H}^1(\Omega).
\]
Then $\Pi$ is a bounded operator. Lemma \ref{lem:Helmholtz_vector} implies that $\overline{V}$, $\Pi$ satisfy the assumption (b) Lemma \ref{lem:FSP} with $V=V^*=L^2(\Omega,\mathbb{V})$, $A=\text{id}$. As a result,
\begin{equation*}
\|\sigma-\sigma_h\|^2\eqsim\|{\rm curl}^*(\sigma-\sigma_h)\|_{H^{-1}(\Omega)}^2+\|\nabla^* (\sigma-\sigma_h)\|_{\widetilde{H}^1(\Omega)^*}^2.
\end{equation*}
Using the orthogonality in Lemma \ref{lem:Helmholtz_vector} and ${\rm curl}^*\sigma={\rm rot}\nabla u=0$, we further verify $c_1=1$ in Lemma \ref{lem:FSP} and obtain
\begin{subequations}\label{eq:sigma_errbound}
    \begin{align}
        &\|\sigma-\sigma_h\|^2\leq\|{\rm curl}^*\sigma_h\|_{H^{-1}(\Omega)}^2+\|\nabla^* (\sigma-\sigma_h)\|_{\widetilde{H}^1(\Omega)^*}^2,\\
        &\|{\rm curl}^*\sigma_h\|_{H^{-1}(\Omega)}^2+\|\nabla^* (\sigma-\sigma_h)\|_{\widetilde{H}^1(\Omega)^*}^2\lesssim\|\sigma-\sigma_h\|^2.
    \end{align}
\end{subequations}
For $v\in H_0^1(\Omega)$, the localization $({\rm curl}^*\sigma_h)_i$ is
\begin{align*}
\langle ({\rm curl}^*\sigma_h)_i,v\rangle&=(\sigma_h,{\rm curl}(\phi_iv))\\
&=(\sigma_h\cdot{\rm curl}\phi_i,v)+(\sigma_h\phi_i,{\rm curl}v)\\
&=(\sigma_h\cdot{\rm curl}\phi_i,v)+\langle{\rm div}(\sigma_h^\perp\phi_i),v\rangle,
\end{align*}
where the algebraic identity \eqref{eq:div_perp} is used.
Therefore, Problem \ref{prob:min_sigmaRT} with $R={\rm curl}^*\sigma_h$ is equivalent to 
\begin{equation*}
\begin{aligned}
&\sigma_i=\arg\min_{\tau\in\mathcal{RT}_p(\Omega_i)}\|\tau+\sigma_h^\perp\phi_i\|_{L^2(\Omega_i)},\\
&\text{ subject to }{\rm div}\tau=\sigma_h\cdot{\rm curl}\phi_i.    
\end{aligned}
\end{equation*}
It is noted that ${\rm div}(\sigma_i+\sigma_h^\perp\phi_i)=({\rm curl}^*\sigma_h)_i.$
Then using Lemma \ref{lem:equi_residualH1}, we obtain
\begin{equation}\label{eq:curlsigmah}
\Big\|\sum_{i\in\mathcal{I}}\sigma_i+\sigma_h^\perp\phi_i\Big\|\lesssim\Big\|{\rm curl}^*\sigma_h\Big\|_{H^{-1}(\Omega)}\leq\Big\|\sum_{i\in\mathcal{I}}\sigma_i+\sigma_h^\perp\phi_i\Big\|.
\end{equation}
To estimate the term $\|\nabla^* (\sigma-\sigma_h)\|_{\widetilde{H}^1(\Omega)^*}$ in \eqref{eq:sigma_errbound}, we proceed as follows
\begin{equation}\label{eq:gradstar}
\begin{aligned}
\|\nabla^* (\sigma-\sigma_h)\|_{\widetilde{H}^1(\Omega)^*}&=\sup_{v\in\widetilde{H}^1(\Omega), |v|_{H^1(\Omega)}=1}(\nabla^* (\sigma-\sigma_h),v)\\    &=\sup_{v\in\widetilde{H}^1(\Omega),|v|_{H^1(\Omega)}=1}(f-f_h,v-v_h)\\
    &\leq\|\gamma h(f-f_h)\|\sup_{v\in\widetilde{H}^1(\Omega),|v|_{H^1(\Omega)}=1}\|\gamma^{-1}h^{-1}(v-v_h)\|.
\end{aligned}
\end{equation}
Combining \eqref{eq:sigma_errbound}, \eqref{eq:curlsigmah} and \eqref{eq:gradstar} completes the proof.
\end{proof}
Recall the Poincar\'e inequality $\|v-v_h\|_{L^2(T)}\leq  (h_T/\pi)|v|_{H^1(T)}$ (cf.~\cite{PayneWeinberger1960,Bebendorf2003}) and thus the estimate $\gamma\leq1/\pi$ holds.

For the mixed method \eqref{def:MFEM_Poisson} for Poisson's equation, a posteriori error estimates by equilibration have been extensively investigated in the literature, see, e.g., \cite{Ainsworth2007,ErnVohralik2015}. Classical equilibrated a posteriori error bounds of mixed FEMs were often achieved
by postprocessing the displacement variable $u_h$. The error estimator in Theorem \ref{thm:mixPoisson} is new since it still utilizes equilibration in $H^{-1}(\Omega)$ by solving local problems with $\sigma_h$ as source terms instead of postprocessing $u_h$.

\subsection{Hellan--Herrmann--Johnson Method}\label{subsect:HHJ} The analysis in Section \ref{subsect:mixPoisson} could be generalized to more complicated model problems such as  the following biharmonic equation in mixed form:
\begin{subequations}\label{eq:mix4thorder}
\begin{align}
\sigma-{\rm hess}u&=0\quad\text{ in }\Omega,\label{eq:mix4thorder_1}\\
    {\rm divdiv}\sigma&=f\quad\text{ in }\Omega,\\
    u=\partial_{nn}u&=0\quad\text{ on }\partial\Omega.
\end{align}
\end{subequations}
A natural finite element discretization for \eqref{eq:mix4thorder} is the Hellan--Herrmann--Johnson (HHJ) method.
Given a piecewise smooth and symmetric stress field $\tau$, let $\tau_{nn}=(\tau n)\cdot n$ and $\tau_{nt}=(\tau n)\cdot t$ denote its normal-normal and normal-tangential component on $\mathcal{E}_h$, respectively.
For a non-negative integer $p$, the HHJ method uses normal-normal continuous stress space 
\begin{align*}
V_h({\rm divdiv})=&\{\tau_h\in L^2(\Omega,\mathbb{S}): \tau_h|_T\in\mathcal{P}_p(\mathbb{S})\text{ for each }T\in\mathcal{T}_h,\\
&(\tau_h)_{nn}\text{ is continuous across each }E\in\mathcal{E}_h,~(\tau_h)_{nn}|_{\partial\Omega}=0\}.  
\end{align*}
The displacement space $V_h({\rm grad})\subset H_0^1(\Omega)$ consists of continuous and piecewise polynomials of degree $\leq p+1$. 

Given $\tau_h\in V_h({\rm divdiv})$ and $v\in \widetilde{H}^2(\Omega)$, direct calculation shows
\begin{equation*}
\langle{\rm divdiv}\tau_h,v\rangle=(\tau_h,{\rm hess}v)=-({\rm div}_h\tau_h,\nabla v)-(\llbracket(\tau_h)_{nt}\rrbracket,\partial_tv)_{\mathcal{E}_h}.
\end{equation*}
For $v\in \widetilde{H}^2(\Omega)+V_h({\rm grad})$, define the bilinear form
\begin{equation}\label{def:b}
\begin{aligned}
    b(\tau_h,v)&=-({\rm div}_h\tau_h,\nabla v)-(\llbracket(\tau_h)_{nt}\rrbracket,\partial_tv)_{\mathcal{E}_h}\\
    &=({\rm div}_h{\rm div}_h\tau_h,v)+(\llbracket n\cdot{\rm div}_h\tau_h\rrbracket,v)_{\mathcal{E}_h}-(\llbracket(\tau_h)_{nt}\rrbracket,\partial_tv)_{\mathcal{E}_h}.
\end{aligned}
\end{equation}
The HHJ method for \eqref{eq:mix4thorder} seeks $\sigma_h\in V_h({\rm divdiv})$ and $u_h\in V_h({\rm grad})$ such that
\begin{subequations}\label{def:HHJ}
\begin{align}
(\sigma_h,\tau_h)-b(\tau_h,u_h)&=0,\quad\tau_h\in V_h({\rm divdiv}),\label{def:HHJ1}\\
b(\sigma_h,v_h)&=(f,v_h),\quad v_h\in V_h({\rm grad}).
\end{align}
\end{subequations}

A posteriori error analysis of \eqref{def:HHJ} relies on the following  Helmholtz-type decomposition, see \cite{BeiraoNiiranenStenberg2010,BeiraoNiiranenStenberg2007} for the original proof and a systematic derivation in \cite{ArnoldHu2021} based on the BGG complex.
\begin{lemma}[Helmholtz Decomposition of Tensor Fields]\label{lem:Helmholtz_tensor}
Let $\Omega\subset\mathbb{R}^2$ be a simply-connected Lipschitz domain. For any $\tau\in L^2(\Omega,\mathbb{S})$, there exist $\varphi\in H_0^1(\Omega,\mathbb{V})$ and $\psi\in\widetilde{H}^2(\Omega)$ such that 
\begin{align*}
\tau={\rm symcurl}\varphi\oplus{\rm hess}\psi.
\end{align*} 
\end{lemma}

Let $P_{h,p}$ be the $L^2$ projection onto the broken finite element space of degree $\leq p$. Let $I_h=I_{h,p+1}$ be the moment-based modified Lagrange interpolation onto $C^0$ piecewise polynomials of degree $\leq p+1$. On each element $T\in\mathcal{T}_h$, the interpolant $I_hv|_T=I_Tv=I_{T,p+1}v\in\mathcal{P}_{p+1}$ is determined by 
\begin{subequations}\label{eq:momentIh}
\begin{align}
I_Tv(a_i)&=v(a_i),\quad\text{for each vertex } a_i\in\overline{T},\\
    \int_{E}(I_Tv)wds&=\int_Evwds,\quad\text{for each edge }E\subset\partial T\text{ and } w\in\mathcal{P}_{p-1}(E),\label{eq:moment_edge}\\
    \int_T(I_Tv)qdx&=\int_Tvqdx,\quad \text{for each }q\in\mathcal{P}_{p-2}(T).\label{eq:moment_volume}
\end{align}
\end{subequations}
In the next theorem, we present a new equilibrated residual error estimator for the HHJ method. 
\begin{theorem}\label{thm:HHJ}
Assume $\Omega$ is a simply-connected polygonal domain. Let $\sigma_h\in V_h({\rm divdiv})$ be a piecewise polynomial of degree $\leq p$. For each $i\in\mathcal{I}$, let $\sigma_i\in\mathcal{RT}_p(\Omega^h_i,\mathbb{V})$ and $u_i\in\mathcal{P}^{\rm dG}_p(\Omega^h_i,\mathbb{V})$ solve
\begin{equation}\label{eq:local_HHJ}
\begin{aligned}
    (\sigma_i,\tau)_{\Omega_i}+({\rm div}\tau,u_i)_{\Omega_i}&=-(\sigma_h^\perp\phi_i,\tau)_{\Omega_i},\quad\tau\in\mathcal{RT}_p(\Omega^h_i,\mathbb{V}),\\
    ({\rm div}\sigma_i,v)_{\Omega_i}&=(\sigma_h\cdot{\rm curl}\phi_i,v)_{\Omega_i},\quad v\in\mathcal{P}^{\rm dG}_p(\Omega^h_i,\mathbb{V}),
\end{aligned} 
\end{equation}
where $\sigma_h\cdot{\rm curl}\phi_i=((\sigma_h)_1\cdot{\rm curl}\phi_i,(\sigma_h)_2\cdot{\rm curl}\phi_i)$.
Then for \eqref{def:HHJ} it holds that 
\begin{equation*}
\|\sigma-\sigma_h\|^2\leq2\Big\|\sum_{i\in\mathcal{I}}\sigma_i+\sigma_h^\perp\phi_i\Big\|^2+\|\alpha h^2(f-P_{h,p-2}f)\|^2,
\end{equation*}
where $P_{h,-2}=P_{h,-1}=0$ and $\alpha$ is a piecewise constant with $\alpha_T=h_T^{-2}\max_{v\neq0}\|v-I_{T,p+1}v\|_{L^2(T)}/|v|_{H^2(T)}$ for $T\in\mathcal{T}_h$.
In addition,
\begin{equation*}
\Big\|\sum_{i\in\mathcal{I}}\sigma_i+\sigma_h^\perp\phi_i\Big\|\lesssim\|\sigma-\sigma_h\|.
\end{equation*}
\end{theorem}
\begin{proof}
For each interior vertex $a_i\in\mathring{\Omega}$, it follows from \eqref{def:HHJ1} with $\tau_h={\rm symcurl}(\phi_ie_j)$ and $b({\rm symcurl}(\phi_ie_j),u_h)=0$ (cf.~\cite{HHX2011}) that $(\sigma_h\cdot{\rm curl}\phi_i,e_j)_{\Omega_i}=0$, $j=1, 2$, which
ensures the well-posedness of the local problem \eqref{eq:local_HHJ} surrounding $a_i$.

Setting $\overline{V}=H_0^1(\Omega,\mathbb{V})\times \widetilde{H}^2(\Omega)$,  $V=V^*=L^2(\Omega,\mathbb{S})$, $A={\rm id}$, $\Pi=({\rm symcurl},{\rm hess}): \overline{V}\rightarrow L^2(\Omega,\mathbb{S})$  in Lemma \ref{lem:FSP}, we obtain the equivalence by Lemma \ref{lem:Helmholtz_tensor}
\begin{equation*}
\|\sigma-\sigma_h\|^2\eqsim\|{\rm symcurl}^*(\sigma-\sigma_h)\|_{H^{-1}(\Omega,\mathbb{V})}^2+\|{\rm hess}^* (\sigma-\sigma_h)\|_{\widetilde{H}^2(\Omega)^*}^2.
\end{equation*}
Using the orthogonality in Lemma \ref{lem:Helmholtz_tensor}, the Korn's inequality (cf.~\cite{Horgan1995})
\begin{equation*}
|\varphi|_{H^1(\Omega)}^2\leq2\|{\rm symcurl}\varphi\|^2,\quad\varphi\in H_0^1(\Omega,\mathbb{V}),  
\end{equation*}
and ${\rm symcurl}^*\sigma={\rm symcurl}^*{\rm hess}u=0$, we further obtain 
\begin{subequations}\label{eq:sigma_errbound_HHJ}
    \begin{align}
        &\|\sigma-\sigma_h\|^2\leq2\|{\rm symcurl}^*\sigma_h\|_{H^{-1}(\Omega,\mathbb{V})}^2+\|{\rm hess}^* (\sigma-\sigma_h)\|_{\widetilde{H}^2(\Omega)^*}^2,\\
        &\|{\rm symcurl}^*\sigma_h\|_{H^{-1}(\Omega,\mathbb{V})}^2+\|{\rm hess}^* (\sigma-\sigma_h)\|_{\widetilde{H}^2(\Omega)^*}^2\lesssim\|\sigma-\sigma_h\|^2.
    \end{align}
\end{subequations}
By direct calculation, we have
\begin{equation*}
\langle ({\rm symcurl}^*\sigma_h)_i,v\rangle=(\sigma_h\cdot{\rm curl}\phi_i,v)+\langle{\rm rot}(\sigma_h\phi_i),v\rangle,\quad v\in H_0^1(\Omega,\mathbb{V}).
\end{equation*}
For the same reason as in the proof of Theorem \ref{thm:mixPoisson}, we obtain
\begin{equation}\label{eq:symcurlsigmah}
\Big\|\sum_{i\in\mathcal{I}}\sigma_i+\sigma_h^\perp\phi_i\Big\|\lesssim\|{\rm symcurl}^*\sigma_h\|_{H^{-1}(\Omega,\mathbb{V})}\leq\Big\|\sum_{i\in\mathcal{I}}\sigma_i+\sigma_h^\perp\phi_i\Big\|.
\end{equation}

By the definition of \eqref{def:b} and $\sigma={\rm hess}u$, we have 
\begin{equation}\label{eq:b_orthogonal}
b(\sigma-\sigma_h,v_h)=0,\quad v_h\in V_h(\rm grad).    
\end{equation}
In what follows, the second term $\|{\rm hess}^* (\sigma-\sigma_h)\|_{\widetilde{H}^2(\Omega)^*}$ is estimated by 
\begin{equation}\label{eq:hess_0}
\begin{aligned}
\|{\rm hess}^* (\sigma-\sigma_h)\|_{\widetilde{H}^2(\Omega)^*}&=\sup_{v\in\widetilde{H}^2(\Omega), |v|_{H^2(\Omega)}=1}\langle{\rm divdiv}(\sigma-\sigma_h),v\rangle\\
    &=\sup_{v\in\widetilde{H}^2(\Omega), |v|_{H^2(\Omega)}=1}b(\sigma-\sigma_h,v-I_hv)\\
    &=\sup_{v\in\widetilde{H}^2(\Omega), |v|_{H^2(\Omega)}=1}(f-{\rm div}_h{\rm div}_h\sigma_h,v-I_hv).
\end{aligned}
\end{equation}
In the last equality, we used the fact that $(\llbracket n\cdot{\rm div}_h\tau_h\rrbracket,v-I_hv)_{\mathcal{E}_h}=0$ and $(\llbracket(\tau_h)_{nt}\rrbracket,\partial_t(v-I_hv))_{\mathcal{E}_h}=0$. It then follows from \eqref{eq:hess_0} and \eqref{eq:moment_volume} that 
\begin{equation}\label{eq:hess}
\begin{aligned}
&\|{\rm hess}^* (\sigma-\sigma_h)\|_{\widetilde{H}^2(\Omega)^*}=\sup_{v\in\widetilde{H}^2(\Omega), |v|_{H^2(\Omega)}=1}(f-P_{h,p-2}f,v-I_hv)\\
&\quad\leq\|\alpha h^2(f-P_{h,p-2}f)\|\sup_{v\in\widetilde{H}^2(\Omega), |v|_{H^2(\Omega)}=1}\|\alpha^{-1}h^{-2}(v-I_hv)\|.
\end{aligned}
\end{equation}
We conclude the stated result by using \eqref{eq:sigma_errbound_HHJ}, \eqref{eq:symcurlsigmah}, \eqref{eq:hess}.
\end{proof}
\begin{remark}\label{remark:HHJ_RTp1} The same analysis implies that the error estimator in Theorem \ref{thm:HHJ} with $\sigma_i\in\mathcal{RT}_{p+1}(\Omega_i,\mathbb{V})$ and $u_i\in\mathcal{P}^{\rm dG}_{p+1}(\Omega_i,\mathbb{V})$ defined by
\begin{equation}\label{eq:local_HHJ_RTp1}
\begin{aligned}
    (\sigma_i,\tau)_{\Omega_i}+({\rm div}\tau,u_i)_{\Omega_i}&=-(\sigma_h^\perp\phi_i,\tau)_{\Omega_i},\quad\tau\in\mathcal{RT}_{p+1}(\Omega_i,\mathbb{V}),\\
    ({\rm div}\sigma_i,v)_{\Omega_i}&=(\sigma_h\cdot{\rm curl}\phi_i,v)_{\Omega_i},\quad v\in\mathcal{P}^{\rm dG}_{p+1}(\Omega_i,\mathbb{V})
\end{aligned} 
\end{equation}
is also a $p$-robust a posteriori error estimate for the HHJ method \eqref{def:HHJ}.
Intuitively, the error estimator based on \eqref{eq:local_HHJ_RTp1} would be more accurate than \eqref{eq:local_HHJ} at the expense of higher computational cost. Numerical results about the error estimator based on \eqref{eq:local_HHJ_RTp1} are presented in Section \ref{subsect:HHJ_experiment}.
\end{remark}

Residual and recovery a posteriori error estimates for the HHJ method could be found in, e.g., \cite{HHX2011,Li2021JSCb}. 
Moreover, 
\cite{BraessPechsteinSchoberl2020} derived an equilibrated a posteriori error estimate of the $C^0$ discontinuous Galerkin method for the fourth order elliptic equation and provided a remark on equilibrated error estimator for the HHJ method. However, a posteriori error estimates in \cite{BraessPechsteinSchoberl2020} are not theoretically  confirmed to be $p$-robust.

It was mentioned in \cite{BraessPechsteinSchoberl2020} that 
\begin{equation*}
    \|v-I_hv\|_{L^2(T)}\leq 0.3682146h_T^2|v|_{H^2(T)}.
\end{equation*}
Therefore, $0.3682146$ is a practical upper bound for the constant $\alpha$ in Theorem \ref{thm:HHJ}.
To obtain tighter interpolation error estimates for larger $p$, one could solve the maximization problem 
\begin{equation}\label{eq:maximization}
\alpha_{\widehat{T}}=\arg\max_{|v|_{H^2(\widehat{T})}=1}\|v-I_{\widehat{T},p+1}v\|_{L^2(\widehat{T})}    
\end{equation}
for each degree $p$ on a reference element $\widehat{T}={\rm conv}\{(0,0), (1,0), (0,1)\}$. Let $F_T(\hat{x})=B_T\hat{x}+b_T$ be the affine transformation from $\widehat{T}$ to $T\in\mathcal{T}_h$ with $B_T=(b_{T,ij})_{1\leq i,j\leq2}$. By a homogeneity argument we obtain a sharper interpolation error bound $\|v-I_Tv\|_{L^2(T)}\leq \alpha_{\widehat{T}}\beta_T|v|_{H^2(T)}$ with $\beta_T^2:=(b_{T,11}^2+b_{T,21}^2)^2+(b_{T,12}^2+b_{T,22}^2)^2+2(b_{T,11}b_{T,12}+b_{T,21}b_{T,22})^2+2b^2_{T,11}b_{T,22}^2+2b^2_{T,12}b_{T,21}^2$ and thus tighter guaranteed a posteriori error estimate of \eqref{def:HHJ}:
\begin{equation}\label{eq:HHJ_practical}
\|\sigma-\sigma_h\|^2\leq2\Big\|\sum_{i\in\mathcal{I}}\sigma_i+\sigma_h^\perp\phi_i\Big\|^2+\sum_{T\in\mathcal{T}_h}\alpha_{\widehat{T}}^2\beta_T^2\|f-P_{h,p-2}f\|_{L^2(T)}^2.
\end{equation}

\begin{remark}
To obtain guaranteed upper bounds for the HHJ method under other boundary conditions, it is necessary to specify the 
constant in Korn's inequalities (cf.~\cite{Horgan1995}). However, the $p$-robustness of the error estimator in Theorem \ref{thm:HHJ} remains true for other boundary conditions and even for space dimension 3. 
\end{remark}

\begin{remark}
The analysis in this section directly yields similar error estimators for $H(\rm divdiv)$ conforming FEMs (cf.~\cite{HuMaZhang2021,ChenHuang2022SINUM,HuLiangMa2022}) for fourth order elliptic equations. Since $H(\rm divdiv)$ conforming finite elements are based on higher order polynomials, the $p$-robustness property is particularly favorable in this scenario. 
\end{remark}

\section{Numerical Experiments}\label{sect:numerical}
To validate the efficiency and robustness of the theoretical results, we test the performance of the equilibrated residual error estimators stated in Theorems \ref{thm:Hcurl}, \ref{thm:curlcurl} and \ref{thm:HHJ} (denoted by $\eta_h$ in each figure). Adaptive mesh refinement is based on standard D\"orfler marking strategy with marking parameter $\theta=0.4$ and the newest vertex bisection. The order $r$ of convergence of any FEM error \texttt{err} is computed by least-squares fitting such that $\texttt{err}\approx N^{-r}$, where $N$ is the number of unknowns in the finite element discretization. In experiments, we also compute the effectivity ratio $\eta_h/\texttt{err}$ for each a posteriori error estimate $\eta_h$.

\subsection{H(curl) Conforming Method}\label{subsect:Hcurl_experiment}We consider two-dimensional $H({\rm curl})$ problems defined on the L-shaped domain $\Omega=(-1,1)^2\backslash([0,1]\times[-1,0])$. Let  $$u(x,y)=\nabla\big(\phi(r)r^{2/3}\sin(2\theta/3)\big)+(\sin(\pi y),\sin(\pi x))$$ be the exact solution of the $H(\rm curl)$ elliptic problem \eqref{eq:Hcurl_elliptic}, where  $(r,\theta)$ is the polar coordinate near the origin and $\phi(r)$ is the cut-off function \begin{equation*}
\phi(r)=\left\{\begin{aligned}
        &\left(1-\frac{r}{0.5}\right)^6,\quad &&r\leq0.5,\\
        &~0,\quad &&r>0.5.
    \end{aligned}\right.
\end{equation*} 
For the semi-definite curl-curl problem \eqref{eq:curlcurl}, we set $f(x,y)=(\sin(\pi y),\sin(\pi x))$ and a very high order finite element solution $\hat{u}\approx u$ as the reference solution  for approximately computing the FEM error $\|{\rm curl}(u-u_h)\|\approx\|{\rm curl}(\hat{u}-u_h)\|$. The curl operators in \eqref{eq:Hcurl_elliptic_FEM}, \eqref{eq:curlcurl_FEM} and  Theorems \ref{thm:Hcurl} and \ref{thm:curlcurl} are naturally replaced with the rot operator.

\begin{figure}[th]
\centering
\includegraphics[width=8cm]{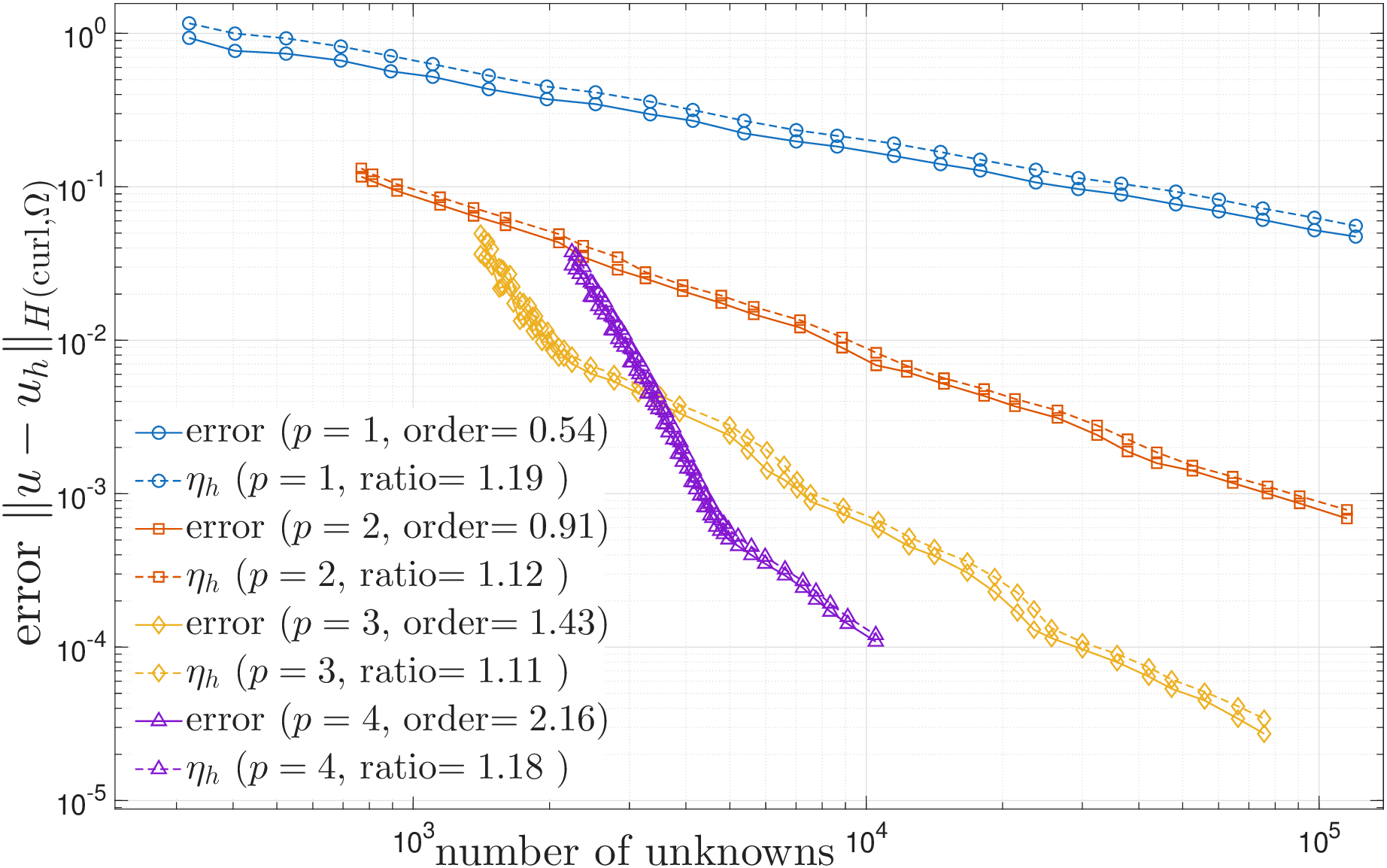}
\caption{A posteriori error estimates and $H(\rm curl)$ norm error of FEMs for the $H(\rm curl)$ elliptic equation.}\label{fig:Hcurl}
\end{figure}

\begin{figure}[th]
    \centering
\includegraphics[width=8cm]{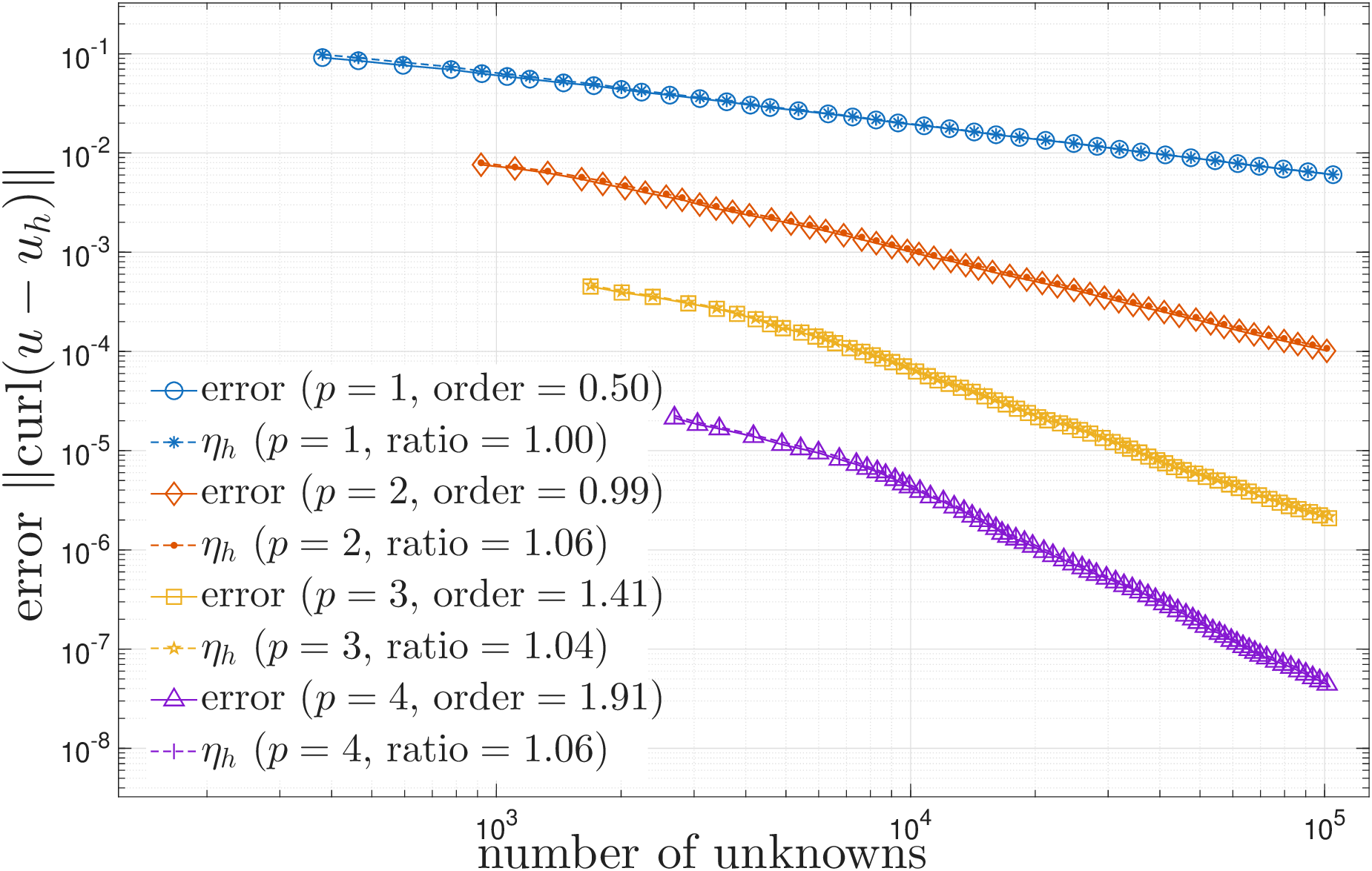} 
\caption{A posteriori error estimates for the curl-curl  equation.}\label{fig:curlcurl}
\end{figure}

As mentioned in Remark \ref{remark:oscillation}, the data oscillation $\|f-P_{h,p}f\|$ is added to the a posteriori error estimates derived in Theorem \ref{thm:Hcurl} and Theorem \ref{thm:curlcurl}, where the local flux given in \eqref{eq:taui} is replaced by its two-dimensional analogue 
\begin{align*}
    (\tau_i,\tau)_{\Omega_i}+({\rm div}\tau,\tilde{u}_i)_{\Omega_i}&=({\rm mskw}(\phi_i{\rm rot}u_h),\tau)_{\Omega_i},\quad\tau\in\mathcal{RT}_{p+1}(\Omega_i,\mathbb{V}),\\
    ({\rm div}\tau_i,v)_{\Omega_i}&=((f-u_h)\phi_i+(\nabla\phi_i)^\perp{\rm rot}u_h,v)_{\Omega_i},\quad v\in\mathcal{P}^{\rm dG}_{p+1}(\Omega_i,\mathbb{V}),
\end{align*}
and the local flux in \eqref{eq:local_curlcurl} is similarly modified as
\begin{align*}
    (\sigma_i,\tau)_{\Omega_i}+({\rm div}\tau,u_i)_{\Omega_i}&=({\rm mskw}(\phi_i{\rm rot}u_h),\tau)_{\Omega_i},\quad\tau\in\mathcal{RT}_{p+1}(\Omega_i,\mathbb{V}),\\
    ({\rm div}\sigma_i,v)_{\Omega_i}&=(f\phi_i+(\nabla\phi_i)^\perp{\rm rot}u_h,v)_{\Omega_i},\quad v\in\mathcal{P}^{\rm dG}_{p+1}(\Omega_i,\mathbb{V}).
\end{align*}

We test corresponding adaptive FEMs  \eqref{eq:Hcurl_elliptic_FEM} and \eqref{eq:curlcurl_FEM} based on the N\'ed\'elec finite element of the second kind of degree $\leq p$. It is shown in Figures \ref{fig:Hcurl} and \ref{fig:curlcurl} that the proposed estimators lead to optimally convergent finite element solutions. The effectiveness ratio remains close to 1 as the polynomial degree grows. Although guaranteed a posteriori error upper bounds are confirm only on convex domains in Corollary \ref{cor:Hcurl} and Theorem \ref{thm:curlcurl}, the equilibrated residual error estimators based auxiliary space still provide rigorous a posteriori error upper bounds even on the L-shaped domain.

\subsection{HHJ Method}\label{subsect:HHJ_experiment} According to approximate optimization \eqref{eq:maximization} on the reference triangle, the constant $\alpha_{\widehat{T}}$ in Section \ref{subsect:HHJ} is computed as follows:  
\begin{align*}
\alpha_{\widehat{T}}&=0.16725~~(p=0),\quad \alpha_{\widehat{T}}=0.04011~~(p=1),\\ 
\alpha_{\widehat{T}}&=0.01973~~(p=2),\quad \alpha_{\widehat{T}}=0.01182~~(p=3).
\end{align*}
On the L-shaped polygon as in Section \ref{subsect:Hcurl_experiment}, we set $f=1$ in the fourth order equation \eqref{eq:mix4thorder} and use a very high order finite element solution $\hat{\sigma}\approx \sigma$ by the HHJ method to approximately compute the stress error $\|\sigma-\sigma_h\|\approx\|\hat{\sigma}-\sigma_h\|$.

\begin{figure}[th]
    \centering
\includegraphics[width=5.5cm]{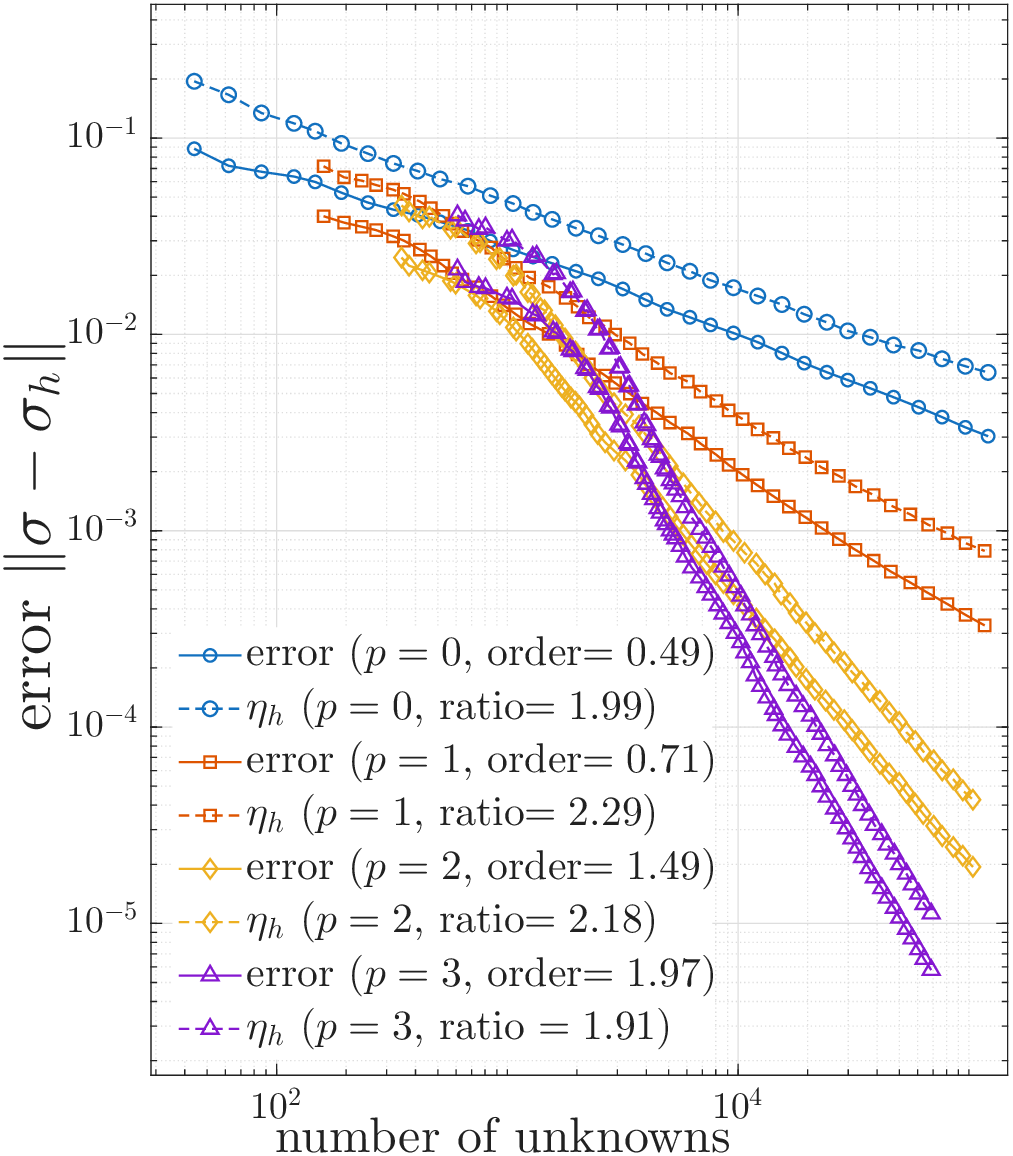} 
\includegraphics[width=5.5cm]{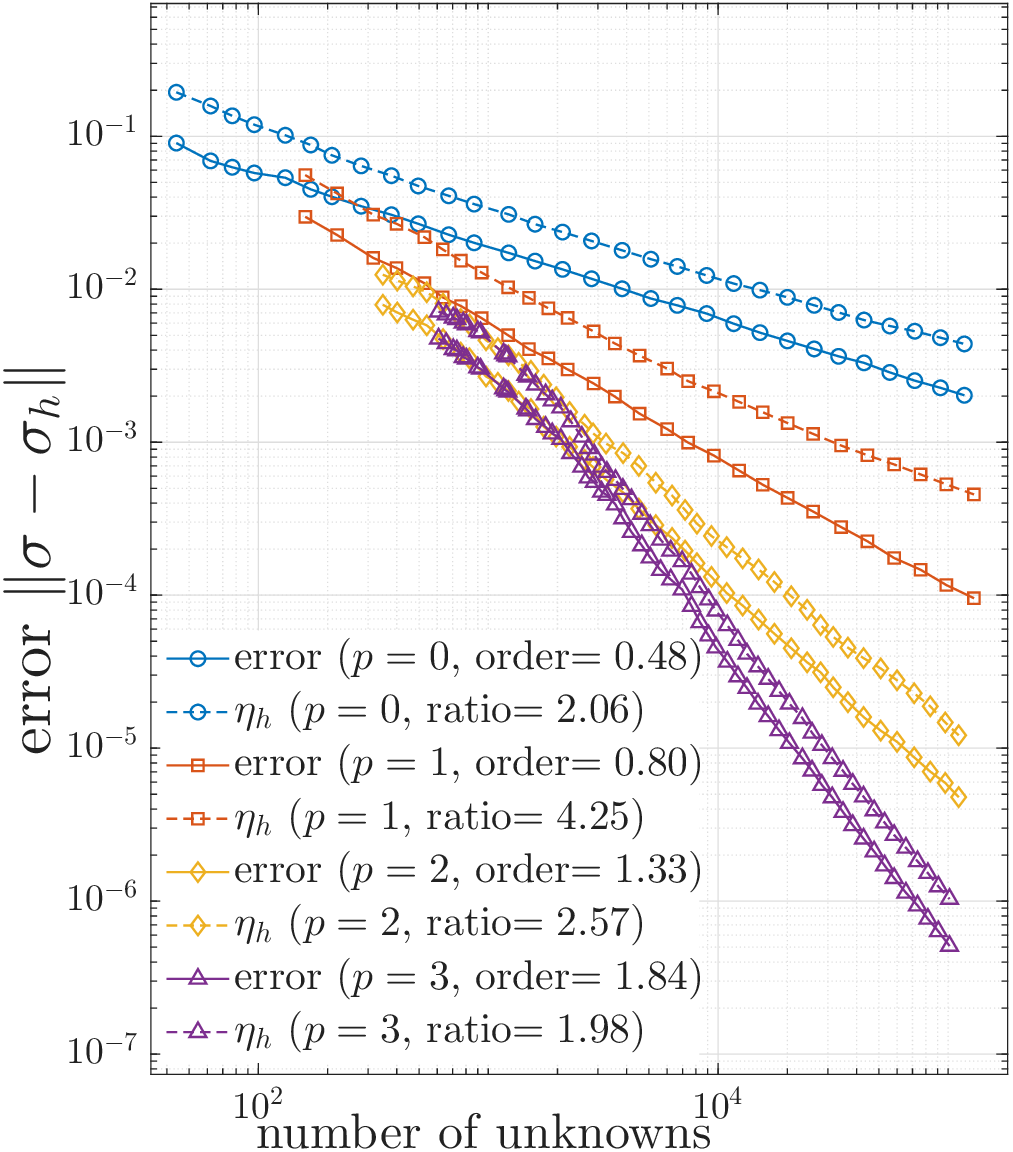} 
    \caption{A posteriori error estimates based on the local $\mathcal{RT}_p\times\mathcal{P}_p^{\rm dG}$ mixed problem for the HHJ method under simply supported (left) and clamped (right) boundary conditions.}
    \label{fig:HHJ_RTp}
\end{figure}

\begin{figure}[th]
    \centering
\includegraphics[width=5.5cm]{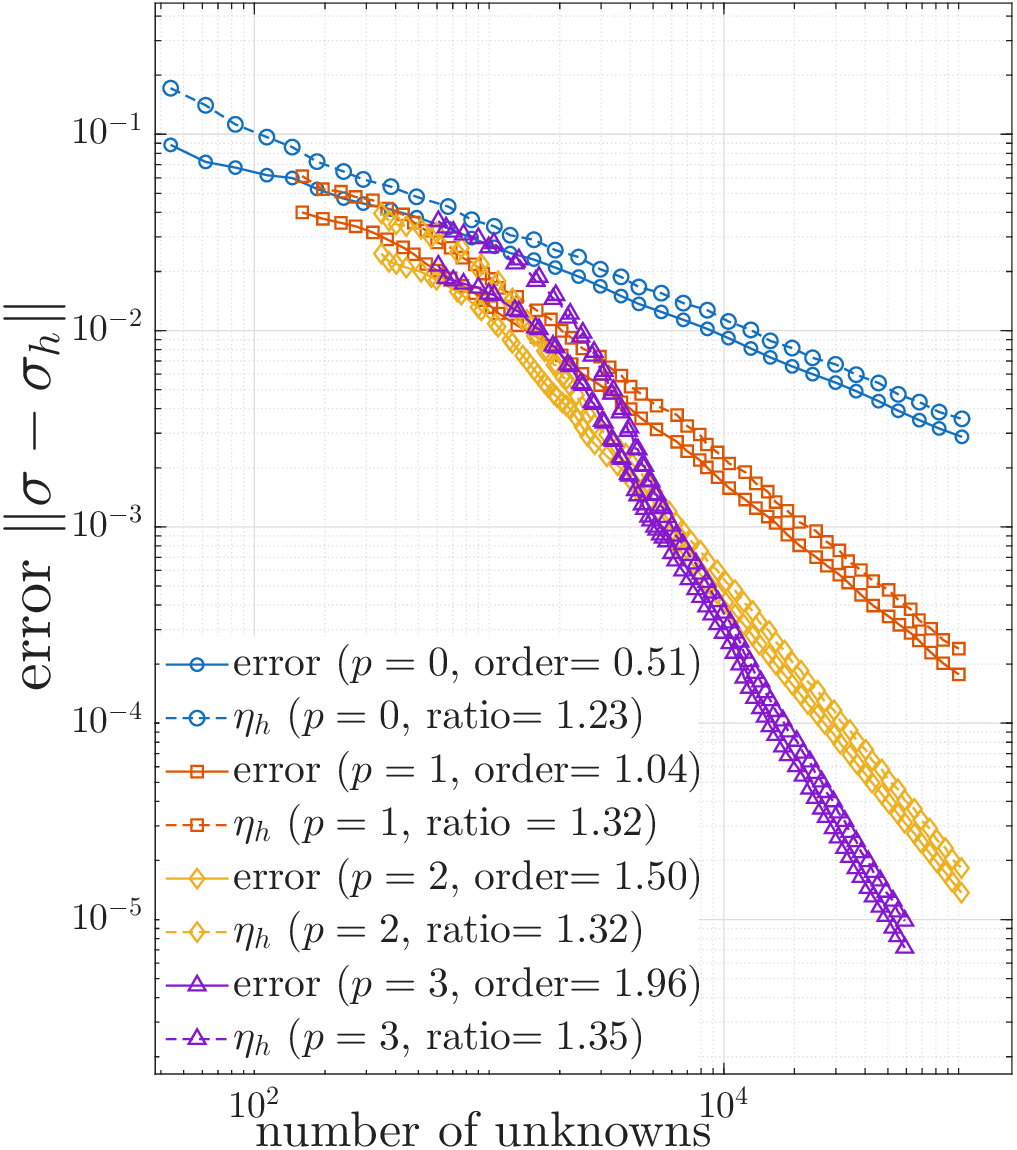} 
\includegraphics[width=5.5cm]{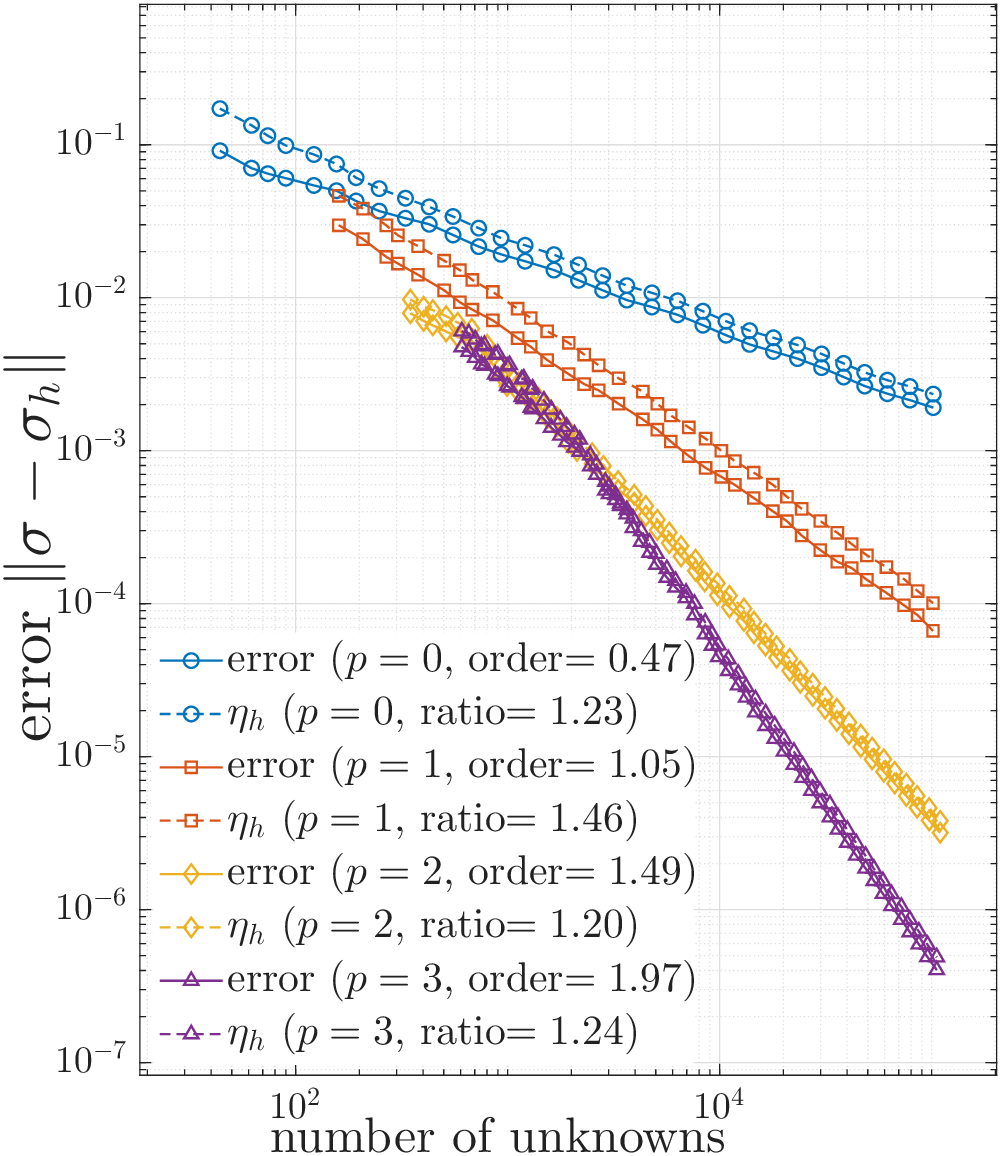} 
    \caption{A posteriori error estimates based on the local $\mathcal{RT}_{p+1}\times\mathcal{P}_{p+1}^{\rm dG}$ mixed problem for the HHJ method under simply supported (left) and clamped (right) boundary conditions.}
    \label{fig:HHJ_RTp1}
\end{figure}

Under the simply supported $u|_{\partial\Omega}=\partial_{nn}u|_{\partial\Omega}=0$ and clamped $u|_{\partial\Omega}=\partial_{n}u|_{\partial\Omega}=0$ boundary conditions, 
we test convergence of the adaptive HHJ method \eqref{def:HHJ} of degree $\leq p$, driven by a posteriori error estimates derived in Theorem \ref{thm:HHJ}. It is shown in Figure \ref{fig:HHJ_RTp} that all estimators lead to finite element solutions with nearly optimal convergence rate. It is also observed that the effectiveness ratio remains to be mild as $p$ grows. However, the effectiveness ratio of the equilibrated residual error estimators for the biharmonic equation is worse than those for $H(\rm curl)$ problems in Section \ref{subsect:Hcurl_experiment}.  

To improve the performance of the error estimator for the HHJ method, we increase the degree of the local problem \eqref{eq:local_HHJ} by one and estimate the error of the HHJ method of degree $\leq p$ based on the local $\mathcal{RT}_{p+1}\times\mathcal{P}_{p+1}^{\rm dG}$ mixed problem as explained in Remark \ref{remark:HHJ_RTp1}. In this scenario, the effectiveness ratio of the modified a posteriori error estimates is relatively close to 1, which is similar to the numerical results for $H(\rm curl)$ problems, see Figure \ref{fig:HHJ_RTp1}.

\section{Concluding remarks}\label{sect:conclusion}
This paper systematically derived equilibrated-type a posteriori error estimates by auxiliary spaces from regular and Helmholtz decompositions. The proposed framework produced novel and user-friendly $p$-robust a posteriori error estimates for controlling the natural norm error of  various $H(\rm curl)$ and $H(\rm div)$ conforming methods. Meanwhile, this framework yields several new $p$-robust $L^2$ norm error estimators for the stress variable of the mixed methods for the Poisson and biharmonic equations.

Applying the auxiliary space approach to mixed FEMs (cf.~\cite{ArnoldAwanouWinther2008,PechsteinSchoberl2011,HuZhang2015,PechsteinSchoberl2018,Li2021M2AN,ChenHuang2022MCOM}) for the elasticity equation yields $H^{-2}$ residuals, while $p$-robust a posteriori error estimation for $H^{-2}$ residuals has not been established in the literature.
Similarly,
natural norm a posteriori error estimates related to higher order complexes are beyond the scope of the current paper. For conforming finite element discretizations, the regular decomposition 
\[
H({\rm divdiv},\Omega,\mathbb{S})={\rm symcurl}H^1(\Omega,\mathbb{V})+H^2(\Omega,\mathbb{S})
\]
of $H({\rm divdiv},\mathbb{S})$ implies equilibration in $H^2(\Omega,\mathbb{S})^*$, a highly complicated subject. Recently, regular decompositions for distributional finite element discretizations was derived in \cite{GopalakrishnanHuSchoberl2025}.
In future work, we will explore  equilibrated residual error estimators for controlling natural norm errors of distributional finite element discretizations.

\section*{Acknowledgments} The author would like to thank Kexin Ding for the help with numerical experiments and Dr.~Kaibo Hu for helpful discussions on distributional finite elements, BGG complexes and regular decomposition.

\bibliographystyle{plain}

\end{document}